\documentclass[11pt]{amsart} 
\usepackage{amsmath}
\usepackage[mathscr]{eucal}

\makeatletter
\newtheorem*{rep@theorem}{\rep@title}
\newcommand{\newreptheorem}[2]{%
\newenvironment{rep#1}[1]{
 \def\rep@title{#2 \ref{##1}}%
 \begin{rep@theorem}}%
 {\end{rep@theorem}}}
\makeatother

\newfam\calfam

\font\calten=eusm10
\font\calseven=eusm7
\font\calfive=eusm5
\textfont\calfam=\calten
\scriptfont\calfam=\calseven
\scriptscriptfont\calfam=\calfive
\def\Cal{\fam=\calfam}

\newtheorem{thm}{Theorem}[section]
\newtheorem{lem}[thm]{Lemma}
\newtheorem{Lem}[thm]{Lemma}
\newtheorem{prop}[thm]{Proposition}
\newtheorem{cor}[thm]{Corollary}
\theoremstyle{remark}
\newtheorem*{rem}{Remark}

\addtocounter{section}{0}
\newreptheorem{thm}{Theorem}

\newcommand{\nmid}{~|\!\!\!^{_{_{\;/}}}~}

\newcommand{\B}{\mathit{b}{\Cal O}}
\renewcommand{\j}{j}
\renewcommand{\P}{\frak p}

\newcommand{\q}{q}

\renewcommand{\u}{u}

\begin{document}

\title[Bounded generation of SL$_2(\Cal O$$_S$)]{Bounded generation of
SL$_2$ over rings of $S$-integers with infinitely many units}

\author[A.~Morgan]{Aleksander V. Morgan}

\author[A.~Rapinchuk]{Andrei S. Rapinchuk}

\author[B.~Sury]{Balasubramanian Sury}

\begin{abstract}
Let $\mathscr{O}$ be the ring of $S$-integers in a number field $k$. We prove that if the group of units $\mathscr{O}^{\times}$ is infinite then every matrix in $\Gamma = \mathrm{SL}_2(\mathscr{O})$ is a product of at most 9 elementary matrices. This essentially completes a long line of research in this direction. As a consequence, we obtain a new proof of the fact that $\Gamma$ is boundedly generated as an abstract group that uses
only standard results from algebraic number theory.
\end{abstract}

\address{Department of Mathematics, University of Virginia,
Charlottesville, VA 22904-4137, USA}

\email{avo2t@eservices.virginia.edu}

\email{asr3x@virginia.edu}

\address{Stat-Math Unit, Indian Statistical Institute,
8th Mile Mysore Road, Bangalore 560059, India}

\email{surybang@gmail.com}
\maketitle

\begin{center}
{\it\large To Alex Lubotzky on his 60th birthday}
\end{center}

\section{Introduction}
Let $k$ be a number field.  Given a finite subset $S$
of the set $V^k$ of valuations of $k$
containing the set $V_{\infty}^k$ of archimedian valuations, we let
$\mathscr{O}_{k , S}$ denote the ring of $S$-integers in $k$, i.e.
$$
\mathscr{O}_{k , S} = \{ a \in k^{\times} \: \vert \: v(a) \geq 0 \
\ \text{for all} \ \ v \in V^k \setminus S \} \cup \{ 0 \}.
$$
As usual, for any commutative ring $R$, we let $\mathrm{SL}_2(R)$ denote the group of unimodular
2 $\times$ 2-matrices over $R$ and refer to the matrices $$
E_{12}(a) = \left( \begin{array}{cc} 1 & a \\
0 & 1 \end{array} \right) \ \ \text{and} \ \ E_{21}(b) = \left(
\begin{array}{cc} 1 & 0 \\ b & 1 \end{array} \right) \ \ \in
\mathrm{SL}_2(R) \ \ \ (a , b \in R) $$
as \emph{elementary} (over $R$).

It was established in \cite{Vas} (see also \cite{Liehl-1}) that if the ring of $S$-integers
$\Cal O = \Cal O$$_{k , S}$ has infinitely many units, the group $\Gamma={\rm SL}_2(\mathscr{O})$
is generated by elementary matrices.  The goal of this paper is to prove that in this case
$\Gamma$ is actually \emph{boundedly} generated by elementaries. More precisely, we prove the following.

\begin{thm}\label{T:Main}
Let $\Cal O=\Cal O$$_{k,S}$ be the ring of $S$-integers in a number field $k$, and assume that
the group of units $\Cal O^\times$ is infinite.  Then every matrix in $\rm SL_2(\Cal O)$ is a
product of at most 9 elementary matrices.
\end{thm}


The quest to validate the property that every element of $\mathrm{SL}_2(\Cal O)$ is a product of
a bounded number of elementary matrices has a considerable history.  First, G.~Cooke and
P.~J.~Weinberger \cite{CW} established it (with the same bound as in Theorem 1.1) assuming
the truth of a suitable form of the Generalized Riemann Hypothesis, which still remains unproven.
Later, it was shown in \cite{LM} (see also \cite{M}) by analytic tools
that the argument can be made unconditional if $| S | \geq \max(5 , 2[k : \Bbb Q]-3)$.
On the other hand, B.~Liehl \cite{Liehl-2} proved the result by algebraic methods for some special
fields $k$.  The first unconditional proof in full generality was given by D.~Carter, G.~Keller
and E.~Paige in an unpublished preprint; their argument was streamlined and made available to the
public by D.~W.~Morris \cite{DWM}.  This argument is based on model theory and provides no
explicit bound on the number of elementaries required; besides, it uses difficult results from
additive number theory.

In \cite{Vs}, M.~Vsemirnov proved Theorem 1.1 for $\mathscr{O} = \mathbb{Z}[1/p]$ using
the results of D.~R.~Heath-Brown \cite{H-B} on Artin's Primitive Root Conjecture (thus,
in a broad sense, this proof develops the initial approach of Cooke and Weinberger \cite{CW});
his bound on the number of elementaries required is $\leq 5$.  Subsequently, the third-named
author re-worked the argument from \cite{Vs} to avoid the use of \cite{H-B} in an unpublished
note.  These notes were the beginning of the work of the first two authors that has eventually
led to a proof of Theorem 1.1 in the general case.  It should be noted that our proof uses only
standard results from number theory such as Artin reciprocity and Chebotarev's Density Theorem,
and is relatively short and constructive with an explicit
bound which is independent of the field $k$ and the set $S$.  This, in particular, implies that
Theorem 1.1 remains valid for any {\it infinite} $S$.

The problem of bounded generation (particularly by elementaries) has been considered for
$S$-arithmetic subgroups of algebraic groups other than $\mathrm{SL}_2$.  A few years after
\cite{CW}, Carter and Keller \cite{CK} showed that ${\rm SL}_n(\mathscr{O})$ for $n \geq 3$
is boundedly generated by elementaries for any ring $\Cal O$ of algebraic integers (see \cite{T}
for other Chevalley groups of rank $>1$, and \cite{ER} for isotropic, but nonsplit (or
quasi-split), orthogonal groups).  The upper bound on the number of factors required to write
every matrix in $\rm SL$$_n(\Cal O)$ as a product of elementaries given in \cite{CK} is
$\frac{1}{2}(3n^2-n)+68\Delta-1$, where $\Delta$ is the number of prime divisors of the
discriminant of $k$; in particular, this estimate depends on the field $k$.  Using our
Theorem \ref{T:Main}, one shows in all cases where the group of units $\Cal O^\times$ is infinite,
this estimate can be improved to $\frac{1}{2}(3n^2-n)+4$, hence made independent of $k$ --
see Corollary \ref{T:Ma}.
The situation not covered by this result are when $\Cal O$ is either $\Bbb Z$ or the ring of
integers in an imaginary quadratic field -- see below.  The former case was treated in \cite{Z}
with an estimate $\frac{1}{2}(3n^2-n)+36$, so only in the case of imaginary quadratic fields
the question of the existence of a bound on the number of elementaries independent of the $k$
remains open.

From a more general perspective, Theorem \ref{T:Main} should be viewed as a contribution
to the sustained effort aimed at proving that all higher rank lattices are boundedly generated
as abstract groups.  We recall that a group $\Gamma$ is said to have {\it bounded generation}
(BG) if there exist elements $\gamma_1, \ldots , \gamma_d \in \Gamma$ such that $$
\Gamma = \langle \gamma_1 \rangle \cdots \langle \gamma_d \rangle,
$$
where $\langle \gamma_i \rangle$ denotes the cyclic subgroup generated by $\gamma_i$.
The interest in this property stems from the fact that while being purely combinatorial in nature,
it is known to have a number of far-reaching consequences for the structure and representations
of a group, particularly if the latter is $S$-arithmetic.  For example, under one additional
(necessary) technical assumption, (BG) implies the rigidity of completely reducible complex
representations of $\Gamma$ (known as $SS$-rigidity) -- see \cite{R1}, \cite[Appendix A]{PR-book}.
Furthermore, if $\Gamma$ is an $S$-arithmetic subgroup of an absolutely simple simply connected
algebraic group $G$ over a number field $k$, then assuming the truth of the Margulis-Platonov
conjecture for the group $G(k)$ of $k$-rational points (cf.~\cite[\S 9.1]{PR-book}),
(BG) implies the {\it congruence subgroup property} (i.e. the finiteness of the corresponding
congruence kernel -- see \cite{Lub}, \cite{PR}).  For applications of (BG) to the
Margulis-Zimmer conjecture, see \cite{ShalWill}.  Given these and other implications of (BG),
we would like to point out the following consequence of Theorem \ref{T:Main}.

\vskip2mm

\begin{cor}\label{c1}
Let $\mathscr{O} = \mathscr{O}_{k , S}$ be the ring of $S$-integers,
in a number field $k$.  If the group of units $\mathscr{O}^{\times}$ is infinite, then the group
$\Gamma = \mathrm{SL}_2(\mathscr{O})$ has bounded generation.
\end{cor}

We note that combining this fact with the results of \cite{Lub}, \cite{PR},
one obtains an alternative proof of the centrality of the congruence kernel for
$\mathrm{SL}_2(\mathscr{O})$ (provided that $\mathscr{O}^{\times}$ is infinite),
originally established by J.-P.~Serre \cite{Serre}.  We also note that (BG) of $\rm SL_2(\Cal O)$
is needed to prove (BG) for some other groups -- cf.~\cite{T} and \cite{ER}.

Next, it should be pointed out  that the assumption that the unit group $\mathscr{O}^{\times}$
is infinite is {\it necessary} for the bounded generation of $\mathrm{SL}_2(\mathscr{O})$,
hence cannot be omitted.  Indeed, it follows from Dirichlet's Unit Theorem \cite[\S 2.18]{CF} that
$\mathscr{O}^{\times}$ is finite only when $|S| = 1$ which happens precisely when $S$ is the set
of archimedian valuations in the following two cases:

1) $k = \Bbb Q$ and $\mathscr{O} = \Bbb Z$.  In this case, the group $\mathrm{SL}_2(\Bbb Z)$ is
generated by the elementaries, but has a nonabelian free subgroup of finite index, which prevents
it from having bounded generation.

2) $k = \Bbb Q(\sqrt{-d})$ for some square-free integer $d\geq 1$, and $\mathscr{O}_d$ is
the ring of algebraic integers integers in $k$.
According to \cite{GS}, the group $\Gamma=\mathrm{SL}_2(\mathscr{O}_d)$ has a
finite index subgroup that admits an epimorphism onto a nonabelian
free group, hence again cannot possibly be boundedly generated.
Moreover, P.~M.~Cohn \cite{C} shows that if $d \notin \{1, 2, 3, 7, 11\}$ then $\Gamma$
is not even generated by elementary matrices.

The structure of the paper is the following. In \S 2 we prove an algebraic result about abelian
subextensions of radical extensions of general field -- see Proposition \ref{P:PPP2}.
This statement, which may be of independent interest, is used in the paper to prove
Theorem \ref{T:Key}.  This theorem is one of the number-theoretic results needed in the proof of
Theorem \ref{T:Main}, and it is established in \S 3 along with some other facts from
algebraic number theory.  One of the key notions in the paper is that of $\mathbb{Q}$-split prime:
we say that a prime $\mathfrak{p}$ of a number field $k$ is $\Bbb Q$-split if it is non-dyadic and
its local degree over the corresponding rational prime is 1.  In \S 3,
we establish some relevant for us properties of such primes (see \S 3.1) and
prove for them in \S 3.2 the following (known - see the remark in \S \ref{S:ANT}) refinement of Dirichlet's Theorem from \cite{BMS}.

\vskip2mm

\begin{repthm}{dit}
Let $\mathscr{O}$ be the ring of $S$-integers in a number field $k$ for some finite
$S \subset V^k$ containing $V_{\infty}^k$. If nonzero $a, b \in \mathscr{O}$ are
relatively prime (i.e., $a\mathscr{O} + b\mathscr{O} = \mathscr{O}$) then there exist
infinitely many principal $\Bbb Q$-split prime ideals $\mathfrak{p}$ of $\mathscr{O}$
with a generator $\pi$ such that $\pi \equiv a$ $(\mathrm{mod}\: b\mathscr{O})$ and
$\pi > 0$ in all real completions of $k$.
\end{repthm}

\vskip2mm

\noindent Subsection  3.3 is devoted to the statement and proof of Theorem \ref{T:Key},
which is another key number-theoretic result needed in the proof of Theorem 1.1.
In \S 4, we prove Theorem \ref{T:Main} and Corollary \ref{c1}.  Finally, in \S5
we correct the faulty example from [Vs] of a matrix in $\mathrm{SL}_2(\Bbb Z[1/p])$,
where $p$ is a prime $\equiv 1(\mathrm{mod}\: 29)$, that is not a product of four elementary
matrices -- see Proposition 5.1, confirming thereby that the bound of 5 in [Vs] is optimal.

\vskip2mm

\noindent {\bf Notations and conventions.} For a field $k$, we let $k^{\mathrm{ab}}$ denote
the maximal abelian extension of $k$.  Furthermore, $\mu(k)$ will denote the group of
all roots of unity in $k$; if $\mu(k)$ is finite, we let $\mu$ denote its order.  For $n \geq 1$
prime to $\mathrm{char}\: k$, we let $\zeta_n$ denote a primitive $n$-th root of unity.

In this paper, with the exception of \S2, the field $k$ will be a field of algebraic numbers
(i.e., a finite extension of $\mathbb{Q}$), in which case $\mu(k)$ is automatically finite.
We let $\Cal O$$_k$ denote the ring of algebraic integers in $k$.  Furthermore, we let
$V^k$ denote the set of (the equivalence classes of) nontrivial valuations of $k$, and
let $V^k_{\infty}$ and $V^k_{f}$ denote the subsets of archimedean and nonarchimedean valuations,
respectively.  For any $v \in V^k$, we let $k_v$ denote the corresponding completion;
if $v \in V^k_f$ then $\mathscr{O}_v$ will denote the valuation ring in $k_v$ with
the valuation ideal $\hat{\mathfrak{p}}_v$ and the group of units $U_v = \mathscr{O}_v^{\times}$.

Throughout the paper, $S$ will denote a fixed finite subset of $V^k$ containing $V^k_{\infty}$,
and $\mathscr{O} = \mathscr{O}_{k , S}$ the corresponding ring of $S$-integers (see above).
Then the nonzero prime ideals of $\mathscr{O}$ are in a natural bijective correspondence
with the valuations in $V^k \setminus S$.  So, for a nonzero prime ideal
$\mathfrak{p} \subset \mathscr{O}$ we let $v_{\mathfrak{p}} \in V^k \setminus S$ denote
the corresponding valuation, and conversely, for a valuation $v \in V^k \setminus S$ we let
$\mathfrak{p}_v \subset \mathscr{O}$ denote the corresponding prime ideal (note that
$\mathfrak{p}_v = \mathscr{O} \cap \hat{\mathfrak{p}}_v$).  Generalizing
Euler's $\varphi$-function, for a nonzero ideal $\mathfrak{a}$ of $\mathscr{O}$, we set $$
\phi(\mathfrak{a}) = \vert (\mathscr{O}/\mathfrak{a})^{\times} \vert.
$$
For simplicity of notation, for an element $a\in\Cal O$, $\phi(a)$ will
always mean $\phi(a\Cal O)$.  Finally, for $a \in k^\times$, we let $V(a)=\{\, v \in V^k_f\ \vert \ v(a) \neq 0 \, \}$.

Given a prime number $p$, one can write any integer $n$ in the form $n=p^e\cdot m$,
for some non-negative integer $e$, where $p\nmid m$.  We then call $p^e$ the
\emph{$p$-primary component} of $n$.

\section{Abelian subextensions of radical extensions.}
\renewcommand{\r}{p}
In this section, $k$ is an arbitrary field.  For a prime $\r\neq$ char $k$, we let $\mu(k)_\r$
denote the subgroup of $\mu(k)$, consisting of elements satisfying $x^{\r^d} = 1$ for some
$d\geq 0$.  If this subgroup is finite, we set $\lambda(k)_\r$ to be the non-negative integer
satisfying $\vert \mu(k)_\r \vert = \r^{\lambda(k)_\r}$; otherwise, set $\lambda(k)_\r=\infty$.
Clearly if $\mu(k)$ is finite, then $\mu = \prod_\r \r^{\lambda(k)_\r}$.
For $a\in k^\times$, we write $\sqrt[n]{a}$ to denote an arbitrary root of the polynomial $x^n-a$.

The goal of this section is to prove the following.

\begin{prop}\label{P:PPP2}
Let $n \geq 1$ be an integer prime to $\rm char$ $k$, and let $u \in k^{\times}$ be such that
$u \notin \mu(k)_\r {k^{\times}}^\r$ for all $\r$ $|$ $n$.  Then the polynomial $x^n-u$ is
irreducible over $k$, and for $t=\sqrt[n]{\u}$ we have
$$
k(t)\cap k^{\mathrm{ab}} = k(t^m)\ \ \text{where} \ \
m=\frac{n}{\displaystyle\prod_{\r\mid n}{\rm gcd}(n,\r^{\lambda(k)_\r})},
$$
with the convention that ${\rm gcd}(n, \r^\infty)$ is simply the $\r$-primary component of $n$.
\end{prop}

\medskip

We first treat the case $n = \r^d$ where $\r$ is a prime.

\begin{prop}\label{P:PPP1}
Let $p$ be a prime number $\neq$ {\rm char} $k$, and let $u\in
k^\times\setminus\mu(k)_p(k^\times)^p$.  Fix an integer $d \geq 1$, set $t = \sqrt[\r^d]{u}$. Then
$$
k(t)\cap k^{\mathrm{ab}} = k(t^{\r^{\gamma}}) \ \ \text{where} \ \ \gamma=\max(0,d-\lambda(k)_\r).
$$
\end{prop}


We begin with the following lemma.
\begin{lem}\label{L:LLL1}
Let $p$ be a prime number $\neq$ {\rm char} $k$, and let $u\in k^\times\setminus\mu(k)_\r(k^\times)^\r$.
Set $k_1=k(\sqrt[\r]{u})$. Then

\medskip
{\rm \ (i)} $[k_1:k]=\r$;

{\rm \ (ii)} $\mu(k_1)_\r=\mu(k)_\r$

\vskip2mm

{\rm (iii)} None of the $\sqrt[\r]{u}$ are in $\mu(k_1)_\r(k_1^\times)^\r$.
\end{lem}
\begin{proof}
(i) follows from \cite[Ch.~VI, Theorem 9.1]{La}, as $u\notin(k^\times)^\r$.

(ii): If $\lambda(k)_\r=\infty$, then there is nothing to prove.  Otherwise,
we need to show that for $\lambda=\lambda(k)_\r$, we have $\zeta_{\r^{\lambda+1}}\notin k_1$.
Assume the contrary.  Then, first, $\lambda > 0$.  Indeed, we have a tower of inclusions
$k\subseteq k(\zeta_\r)\subseteq k_1$.  Since $[k_1:k] = \r$ by (i), and $[k(\zeta_p):k]\leq p-1$,
we conclude that $[k(\zeta_p):k] = 1$, i.e.~$\zeta_p\in k$.

Now, since $\zeta_{\r^{\lambda+1}}\notin k$, we have
\begin{equation}\label{E:EEE1}
k_1=k(\zeta_{\r^{\lambda+1}})=k(\sqrt[\r]{\zeta_{\r^{\lambda}}}).
\end{equation}
But according to Kummer's theory (which applies because $\zeta_\r \in k$), the fact that
$k(\sqrt[\r]{a}) = k(\sqrt[\r]{b})$ for $a , b \in k^{\times}$ implies that the images of $a$ and
$b$ in $k^{\times}/(k^{\times})^p$ generate the same subgroup.  So, it follows from (\ref{E:EEE1})
that $u\zeta_\r^i\in(k^{\times})^\r$ for some $i$, and therefore $u\in\mu(k)_\r(k^{\times})^\r$,
contradicting our choice of $u$.

\vskip2mm

(iii): Assume the contrary, i.e.~some $p$-th root $\sqrt[p]{u}$ can be written in the form
$\sqrt[\r]{u} = \zeta a^\r$ for some $a \in k_1^{\times}$ and $\zeta\in\mu(k_1)_\r$.
Let $N = N_{k_1/k} \colon k_1^{\times} \to k^{\times}$ be the norm map. Then $$
N(\sqrt[\r]{u}) = N(\zeta) N(a)^\r.
$$
Clearly, $N(\zeta) \in \mu(k)_\r$, so $N(\sqrt[\r]{u}) \in \mu(k)_\r(k^{\times})^\r$. On the other
hand, $N(\sqrt[\r]{u}) = u$ for $\r$ odd, and $-u$ for $\r = 2$. In all cases, we obtain that
$u \in \mu(k)_\r(k^{\times})^\r$. A contradiction.
\end{proof}

\medskip

A simple induction now yields the following:
\begin{cor}\label{C:CCC1}
Let $p$ be a prime number $\neq$ {\rm char} $k$, and let
$u \in k^{\times} \setminus \mu(k)_\r(k^{\times})^\r$.  For a fixed integer $d\geq 1$, set
$k_d=k(\sqrt[\r^d]{u})$.  Then:

\medskip
{\rm\ \ (i)} $[k_d:k]$ = $\r^d$;
\vskip2mm
{\rm\ (ii)} $\mu(k_d)_\r=\mu(k)_\r$, hence $\lambda(k_d)_\r=\lambda(k)_\r$.
\end{cor}
Of course, assertion (i) is well-known and follows, for example, from \cite[Ch.~VI, \S 9]{La}.

\begin{lem}\label{L:LLL2}
Let $p$ be a prime number $\neq \mathrm{char}\: k$, and let
$u \in k^{\times} \setminus \mu(k)_\r(k^{\times})^\r$.  Fix an integer $d \geq 1$, and set
$t = \sqrt[p^d]{u}$ and $k_d = k(t)$.  Furthermore, for an integer $j$ between $0$ and $d$
define $\ell_j = k(t^{p^{d - j}}) \simeq k(\sqrt[p^j]{u})$.  Then any intermediate subfield
$k \subseteq \ell \subseteq k_d$ is of the form $\ell = \ell_j$ for some $j\in\{ 0,\ldots,d\}$.
\end{lem}
\begin{proof}
Given such an $\ell$, it follows from Corollary \ref{C:CCC1}(i) that $[k_d:\ell]=p^j$ for some
$0\leq j\leq d$.  Since any conjugate of $t$ is of the form $\zeta\cdot t$ where $\zeta^{\r^d}=1$,
we see that the norm $N_{k_d/\ell}(t)$ is of the form $\zeta_0 t^{\r^\j}$, where again
$\zeta_0^{\r^d}=1$.  Then $\zeta_0\in\mu(k_d)_\r$, and using Corollary \ref{C:CCC1}(ii),
we conclude that $\zeta_0\in k\subseteq\ell$.  So, $t^{\r^\j}\in\ell$, implying the inclusion
$\ell_{d-j} \subseteq \ell$.  Now, the fact that $[k_d : \ell_{d-j}] = p^j$ implies that
$\ell = \ell_{d-j}$, yielding our claim.
\end{proof}

\vskip5mm

\noindent {\it Proof of Proposition \ref{P:PPP1}.}
Set $\lambda = \lambda(k)_\r$.  Then for any $d \leq \lambda$ the extension $k(\sqrt[p^d]{u})/k$
is abelian, and our assertion is trivial.  So, we may assume that $\lambda < \infty$ and
$d > \lambda$.  It follows from Lemma \ref{L:LLL2} that $\ell := k(t) \cap k^{\mathrm{ab}}$ is
of the form $\ell_{d - j} = k(t^{p^j})$ for some $j \in \{ 0, \ldots , d \}$.  On the other hand,
$\ell_{d-j}/k$ is a Galois extension of degree $p^{d-j}$, so must contain the conjugate
$\zeta_{p^{d-j}} t^{p^{d-j}}$ of $t^{p^{d-j}}$, implying that $\zeta_{p^{d-j}} \in \ell_{d-j}$.
Since $\ell_{d-j} \simeq k(\sqrt[p^{d-j}]{u})$, we conclude from Corollary \ref{C:CCC1}(ii) that
$d - j \leq \lambda$, i.e.~$j \geq d - \lambda$.  This proves the inclusion
$\ell \subseteq k(t^{p^{\gamma}})$; the opposite inclusion is obvious.

\vskip5mm

\noindent {\it Proof of Proposition \ref{P:PPP2}.}
Let $n = p_1^{\alpha_1} \cdots p_s^{\alpha_s}$ be the prime factorization of $n$, and
for $i = 1, \ldots , s$ set $n_i = n/p_i^{\alpha_i}$.  Let $t = \sqrt[n]{u}$ and $t_i = t^{n_i}$
(so, $t_i$ is a $p_i^{\alpha_i}$-th root of $u$).  Using again \cite[Ch.~VI, Theorem 9.1]{La} we conclude
that $[k(t) : k] = n$, which implies that
\begin{equation}\label{i}
[k(t) : k(t_i)] = n_i \ \ \text{for all} \ \ i = 1, \ldots , r.
\end{equation}
Since for $K := k(t) \cap k^{\mathrm{ab}}$ the degree $[K : k]$ divides $n$, we can write
$K = K_1 \cdots K_s$ where $K_i$ is an abelian extension of $k$ of degree $p_i^{\beta_i}$
for some $\beta_i \leq \alpha_i$. Then the degree $[K_i(t_i) : k(t_i)]$ must be a power of $p_i$.
Comparing with (\ref{i}), we conclude that $K_i \subseteq k(t_i)$.  Applying Proposition \ref{P:PPP1}
with $d = \alpha_i$, we obtain the inclusion
\begin{equation}\label{pi}
K_i \subseteq k(t_i^{p_i^{\gamma_i}}) = k(t^{n_i p_i^{\gamma_i}}) \ \ \text{where} \ \
\gamma_i = \max(0 \, , \, \alpha_i - \lambda(k)_{\r_i}).
\end{equation}
It is easy to see that the g.c.d.~of the numbers $n_ip_i^{\gamma_i}$ for $i = 1, \ldots , s$ is
$$
m = \frac{n}{\displaystyle\prod_{\r\mid n}{\rm gcd}(n,\r^{\lambda(k)_\r})}.
$$
Furthermore, the subgroup of $k(t)^{\times}$ generated by
$t^{n_1 p_1^{\gamma_1}}, \ldots , t^{n_s p_s^{\gamma_s}}$ coincides with the cyclic subgroup
with generator $t^m$.  Then (\ref{pi}) yields the following inclusion
$$
K = K_1 \cdots K_s \subseteq k(t^m).
$$
Since the opposite inclusion is obvious, our claim follows. \hfill $\Box$

\vskip5mm

\begin{cor}\label{L:max-ab}
Assume that $\mu = \vert \mu(k)\vert<\infty$.  Let $P$ be a finite set of rational primes $\neq \mathrm{char}\: k$, and
define $$
\mu' = \mu \cdot \prod_{p \in P} p.
$$
Given $u \in k^{\times}$ such that $$
u \notin \mu(k)_\r(k^{\times})^\r \ \ \text{for all} \ \ \r\in P,
$$
for any abelian extension $F$ of $k$ the intersection $$
E := F \cap k\left( \sqrt[\mu']{u} , \zeta_{\mu'} \right)
$$
is contained in $k\left(\sqrt[\mu]{u} , \zeta_{\mu'}\right)$.
\end{cor}
\begin{proof}
Without loss of generality we may assume that $\zeta_{\mu'} \in F$, and then we have the following tower of field extensions
$$
k\left(\sqrt[\mu]{u} \, , \, \zeta_{\mu'}  \right) \subset E\left(\sqrt[\mu]{u}\right) \subset k\big( \sqrt[\mu']{u} \, , \, \zeta_{\mu'} \big).
$$
We note that the degree $\left[k\left( \sqrt[\mu']{u} \, , \, \zeta_{\mu'} \right) :    k\left(\sqrt[\mu]{u} \, , \, \zeta_{\mu'}  \right) \right]$ divides $\prod_{p \in P} p$. So, if we assume that the assertion of the lemma is false, then we should be able to find to find a prime $p \in P$ that divides the degree $\left[E\left(\sqrt[\mu]{u}\right) :   k\left(\sqrt[\mu]{u} \, , \, \zeta_{\mu'}  \right)\right]$, and therefore does {\it not} divide the degree
$\left[k\left( \sqrt[\mu']{u} \, , \, \zeta_{\mu'} \right) :   E\left(\sqrt[\mu]{u}\right)\right]$. The latter implies that $\sqrt[p\mu]{u} \in E\left(\sqrt[\mu]{u}\right)$. But this contradicts Proposition 2.1 since $E\left(\sqrt[\mu]{u}\right) = E \cdot k\left(\sqrt[\mu]{u}\right)$ is an abelian extension of $k$.

\end{proof}
\vskip5mm
\section{Results from Algebraic Number Theory}\label{S:ANT}

\noindent {\bf 1. $\Bbb Q$-split primes.} Our proof of Theorem \ref{T:Main} heavily relies on properties of so-called $\Bbb Q$-split primes in $\mathscr{O}$.

\vskip1mm

\noindent {\bf Definition.} Let $\mathfrak{p}$ be a nonzero prime ideal of $\mathscr{O}$, and
let $p$ be the corresponding rational prime.  We say that $\mathfrak{p}$ is {\it $\Bbb Q$-split}
if $p > 2$, and for the valuation $v = v_{\mathfrak{p}}$ we have $k_v = \Bbb Q_p$.

\vskip1mm

For the convenience of further references, we list some simple properties of $\Bbb Q$-split primes.
\begin{lem}\label{prim}
Let $\mathfrak{p}$ be a $\Bbb Q$-split prime in $\mathscr{O}$, and for $n \geq 1$ let
$\rho_n\colon\mathscr{O}\to\mathscr{O}/\mathfrak{p}^n$ be the corresponding quotient map. Then:

\vskip2mm

{\rm (a)}
the group of invertible elements $(\mathscr{O}/\mathfrak{p}^n)^{\times}$ is cyclic for any $n$;

\vskip1mm

{\rm (b)} if $c \in \mathscr{O}$ is such that $\rho_2(c)$ generates
$(\mathscr{O}/\mathfrak{p}^2)^{\times}$ then $\rho_n(c)$ generates
$(\mathscr{O}/\mathfrak{p}^n)^{\times}$ for any $n \geq 2$.
\end{lem}
\begin{proof}
Let $p > 2$ be the rational prime corresponding to $\mathfrak{p}$, and
$v = v_{\mathfrak{p}}$ be the associated valuation of $k$.  By definition, $k_v = \Bbb Q_p$,
hence $\mathscr{O}_v = \Bbb Z_p$.  So, for any $n \geq 1$ we will have canonical ring isomorphisms
\begin{equation}\label{E:Iso}
\mathscr{O}/\mathfrak{p}^n \simeq \mathscr{O}_v /\hat{\mathfrak{p}}_v^n \simeq \Bbb Z_p/p^n\Bbb Z_p \simeq \Bbb Z/p^n\Bbb Z.
\end{equation}
Then (a) follows from the well-known
fact that the group $(\Bbb Z/p^n\Bbb Z)^\times$ is cyclic. Furthermore, the isomorphisms in (\ref{E:Iso})
are compatible for different $n$'s.  Since the kernel of the group homomorphism
$(\Bbb Z/p^n\Bbb Z)^{\times} \to (\Bbb Z/p^2\Bbb Z)^{\times}$ is contained in the
Frattini subgroup of $(\Bbb Z/p^n\Bbb Z)^{\times}$ for $n\geq 2$, the same is true for
the homomorphism $(\mathscr{O}/\mathfrak{p}^n)^\times\to(\mathscr{O}/\mathfrak{p}^2)^{\times}$.
This easily implies (b).
\end{proof}

Let $\mathfrak{p}$ be a $\Bbb Q$-split prime, let $v = v_{\mathfrak{p}}$ be the corresponding
valuation.  We will now define the {\it level} $\ell_{\mathfrak{p}}(u)$ of an element
$u\in\mathscr{O}_v^{\times}$ and establish some properties of this notion that we will need later.

Let $p > 2$ be the corresponding rational prime. The group of $p$-adic units $\mathbb{U}_p = \mathbb{Z}_p^{\times}$ has the natural filtration by the congruence subgroups
$$
\mathbb{U}_p^{(i)} = 1 + p^{i}\mathbb{Z}_p \ \ \text{for} \ \ i \in \mathbb{N}.
$$
It is well-known that
$$
\mathbb{U}_p = C \times \mathbb{U}_p^{(1)}
$$
where $C$ is the cyclic group of order $(p-1)$ consisting of all roots of unity in $\mathbb{Q}_p$. Furthermore, the logarithmic map yields a continuous isomorphism $\mathbb{U}_p^{(i)} \to p^i\mathbb{Z}_p$, which implies that for any $u \in \mathbb{U}_p \setminus C$, the closure of the cyclic group
generated by $u$ has a decomposition of the form
$$
\overline{\langle u \rangle} = C' \times \mathbb{U}_p^{(\ell)}
$$
for some subgroup  $C' \subset C$ and some integer $\ell = \ell_p(u) \geq 1$ which we will refer to as the $p$-{\it level} of $u$. We also set $\ell_p(u) = \infty$ for $u \in C$.

Returning now to a $\mathbb{Q}$-split prime $\mathfrak{p}$ of $k$ and keeping the above notations, we define the $\mathfrak{p}$-{\it level} $\ell_{\mathfrak{p}}(u)$ of $u \in \mathscr{O}_v^{\times}$ as the the $p$-level of the element in $\mathbb{U}_p$ that corresponds to $u$ under the natural identification $\mathscr{O}_v = \Bbb Z_p$. We will need the following.
\begin{lem}\label{c9}
Let $\mathfrak{p}$ be a $\Bbb Q$-split prime in $\mathscr{O}$, let $p$ be the corresponding
rational prime, and $v = v_{\frak p}$ the corresponding valuation.  Suppose we are given
an integer $d\geq 1$ not divisible by $p$, a unit $u \in \mathscr{O}_v^{\times}$ of
infinite order having $\mathfrak{p}$-level $s = \ell_{\mathfrak{p}}(u)$, an integer $n_s$, and
an element $c \in \mathscr{O}_v$ such that $u^{n_s} \equiv c\ (\mathrm{mod}\: \mathfrak{p}^s)$.
Then for any $t \geq s$ there exists an integer $n_t \equiv n_s\ (\mathrm{mod}\: d)$ for which
$u^{n_t} \equiv c\ (\mathrm{mod}\: \mathfrak{p}^t)$.
\end{lem}
\begin{proof}
In view of the identification $\mathscr{O}_v = \Bbb Z_p$, it is enough to prove the corresponding
statement for $\Bbb Z_p$. More precisely, we need to show the following: {\it Let $u\in\Bbb U_p$
be a unit of infinite order and $p$-level $s = \ell_p(u)$. If $c \in \mathbb{U}_p$ and
$n_s \in \mathbb{Z}$ are such that $u^{n_s} \equiv c$ $(\mathrm{mod}\: p^s)$, then
for any $t \geq s$ there exists $n_t \equiv n_s (\mathrm{mod}\: d)$ such that $u^{n_t} \equiv c$
$(\mathrm{mod}\: p^t)$.}  Thus, we have that $u^{n_s} \in c \mathbb{U}_p^{(s)}$, and
we wish to show that $$
u^{n_s} \cdot \langle u^d \rangle \bigcap c \mathbb{U}_p^{(t)} \neq \emptyset.
$$
Since $c \mathbb{U}_p^{(t)}$ is open, it is enough to show that
\begin{equation}\label{E:nonempty}
u^{n_s} \cdot \overline{\langle u^d \rangle} \bigcap c \mathbb{U}_p^{(t)} \neq \emptyset.
\end{equation}
But since $\ell_p(u) = s$ and $d$ is prime to $p$, we have the inclusion $\overline{\langle u^d \rangle} \supset \mathbb{U}_p^{(s)}$, and (\ref{E:nonempty}) is obvious.
\end{proof}

\noindent {\bf 2. Dirichlet's Theorem for $\Bbb Q$-split primes.} The following known (see the remark below)
result gives the existence of $\Bbb Q$-split primes in arithmetic progressions.
\begin{thm}\label{dit}
Let $\mathscr{O}$ be the ring of $S$-integers in a number field $k$ for some finite
$S \subset V^k$ containing $V_{\infty}^k$. If nonzero $a, b \in \mathscr{O}$ are
relatively prime (i.e., $a\mathscr{O} + b\mathscr{O} = \mathscr{O}$) then there exist
infinitely many principal $\Bbb Q$-split prime ideals $\mathfrak{p}$ of $\mathscr{O}$
with a generator $\pi$ such that $\pi \equiv a$ $(\mathrm{mod}\: b\mathscr{O})$ and
$\pi > 0$ in all real completions of $k$.
\end{thm}

The proof follows the same general strategy as the proof of Dirichlet's Theorem in [BMS] - see Theorem A.10 in the Appendix on Number Theory. First, we will quickly review some basic facts from global class field theory (cf., for example, [CF], Ch. VII) and fix some notations. Let $J_k$ denote the {\it group of ideles} of $k$ with the natural topology; as usual, we identify $k^{\times}$ with the (discrete) {\it subgroup of principal ideles} in $J_k$. Then for every open subgroup $\mathscr{U} \subset J_k$ of finite index containing $k^{\times}$ there exists a finite abelian Galois extension $L/k$ and a continuous surjective homomorphism $\alpha_{L/k} \colon J_k \to \mathrm{Gal}(L/k)$ (known as the {\it norm residue map}) such that

\vskip2mm

$\bullet$ $\mathscr{U} = \mathrm{Ker}\: \alpha_{L/k} = N_{L/k}(J_L)k^{\times}$;

\vskip1mm

$\bullet$ \parbox[t]{15.5cm}{for every nonarchimedean $v \in V^k$ which is unramified in $L$
we let $\mathrm{Fr}_{L/k}(v)$ denote the Frobenius automorphism of $L/k$ at $v$ (i.e.,
the Frobenius automorphism $\mathrm{Fr}_{L/k}(w \vert v)$ associated to some (equivalently, any)
extension $w \vert v$) and let $\mathbf{i}(v) \in J_k$ be an idele with the components $$
\mathbf{i}(v)_{v'} = \left\{ \begin{array}{lcl} 1 & , & v' \neq v \\ \pi_v & , & v' = v
\end{array}  \right.,
$$
where $\pi_v\in k_v$ is a uniformizer; then $\alpha_{L/k}(\mathbf{i}(v)) = \mathrm{Fr}_{L/k}(v)$.}

\vskip2mm

\noindent For our fixed finite subset $S \subset V^k$ containing $V^k_{\infty}$, we define the following open subgroup of $J_k$:
$$
U_S := \prod_{v \in S} k_v^{\times} \times \prod_{v \in V^k \setminus S} U_v.
$$
Then the abelian extension of $k$ corresponding to the subgroup $\mathscr{U}_S := U_S k^{\times}$ will be called the {\it Hilbert $S$-class field} of $k$ and denoted $K$ throughout the rest of the paper.

Next, we will introduce the idelic $S$-analogs of {\it ray groups}.  Let $\mathfrak{b}$
be a nonzero ideal of $\mathscr{O} = \mathscr{O}_{k , S}$ with the prime factorization
\begin{equation}\label{E:X003}
\mathfrak{b} = \mathfrak{p}_1^{n_1} \cdots \mathfrak{p}_t^{n_t},
\end{equation}
let $v_i = v_{\mathfrak{p}_i}$ be the valuation in $V^k \setminus S$ associated with
$\mathfrak{p}_i$, and let $V(\mathfrak{b}) = \{v_1, \ldots , v_t\}$.  We then define an open subgroup
$$
R_S(\mathfrak{b}) = \prod_{v\in V^k} R_v
$$
where the open subgroups $R_v\subseteq k_v^\times$ are defined as follows.  For $v$ real, we let
$R_v$ be the subgroup of positive elements, letting $R_v=k_v^\times$ for all other $v\in S$, and
setting $R_v=U_v$ for all $v\notin S\cup V(\frak b)$.  It remains to define $R_v$ for
$v=v_i\in V(\frak b)$, in which case we set it to be the congruence subgroup $U_{v_i}^{(n_i)}$ of
$U_{v_i}$ modulo $\hat{\frak p}^{n_i}_{v_i}$.
We then let $K(\mathfrak{b})$ denote the abelian extension of $k$ corresponding to
$\mathbf{R}_S(\mathfrak{b}):= R_S(\mathfrak{b}) k^{\times}$ (``ray class field").
(Obviously, $K(\mathfrak{b})$ contains $K$ for any nonzero ideal $\mathfrak{b}$ of $\mathscr{O}$.)
Furthermore, given $c \in k^{\times}$, we let $\mathbf{j}_{\mathfrak{b}}(c)$ denote
the idele with the following components:
$$
\mathbf{j}_{\mathfrak{b}}(c)_v = \left\{ \begin{array}{lcl}
c & , & v \in V(\mathfrak{b}), \\
1 & , & v \notin V(\mathfrak{b}).
\end{array}    \right.
$$
Then $\theta_{\mathfrak{b}} \colon k^{\times} \to \mathrm{Gal}(K(\mathfrak{b})/k)$ defined by $c \mapsto \alpha_{K(\mathfrak{b})/k}(\mathbf{j}_{\mathfrak{b}}(c))^{-1}$ is a group homomorphism.

\vskip1mm

The following lemma summarizes some simple properties of these definitions.

\begin{lem}\label{dis}
Let $\mathfrak{b} \subset \mathscr{O}$ be a nonzero ideal.

\vskip2mm

\noindent {\rm (a)} \parbox[t]{15.5cm}{If a nonzero $c \in \mathscr{O}$ is relatively prime to $\mathfrak{b}$ (i.e. $c\mathscr{O} + \mathfrak{b} =
\mathscr{O}$) then $\theta_{\mathfrak{b}}(c)$ restricts to the Hilbert $S$-class field $K$ trivially.}

\vskip2mm

\noindent {\rm (b)} \parbox[t]{15.5cm}{If nonzero $c_1 , c_2 \in \mathscr{O}$ are both relatively prime to $\mathfrak{b}$ then $c_1 \equiv c_2\ (\mathrm{mod}\: \mathfrak{b})$ is equivalent to \begin{equation}\label{E:X002}\mathrm{pr}_{\mathfrak{b}}\left(\mathbf{j}_{\mathfrak{b}}(c_1) R_S(\mathfrak{b})\right) = \mathrm{pr}_{\mathfrak{b}}\left(\mathbf{j}_{\mathfrak{b}}(c_2) R_S(\mathfrak{b})\right)\end{equation} where $\displaystyle \mathrm{pr}_{\mathfrak{b}} \colon J_k \to \prod_{v \in V(\mathfrak{b})} k_v^{\times}$ is the natural projection.}

\end{lem}
\begin{proof}
(a): Since $c$ is relatively prime to $\mathfrak{b}$, we have $\mathbf{j}_{\mathfrak{b}}(c) \in U_S$. So, using the functoriality properties of the norm residue map, we obtain
$$
\theta_{\mathfrak{b}}(c) \vert K = \alpha_{K(\mathfrak{b})/k}(\mathbf{j}_{\mathfrak{b}}(c))^{-1} \vert K = \alpha_{K/k}(\mathbf{j}_{\mathfrak{b}}(c))^{-1} = \mathrm{id}_K
$$
because $\mathbf{j}_{\mathfrak{b}}(c) \in U_S \subset \mathscr{U}_S = \mathrm{Ker}\: \alpha_{K/k}$, as required.

\vskip2mm

(b): As above, let (\ref{E:X003}) be the prime factorization of $\mathfrak{b}$, let
$v_i = v_{\mathfrak{p}_i} \in V^k \setminus S$ be the valuation associated with $\mathfrak{p}_i$.
Then for any $c_1 , c_2 \in \mathscr{O}$, the congruence
$c_1 \equiv c_2\ (\mathrm{mod}\: \mathfrak{b})$ is equivalent to
\begin{equation}\label{E:X004}
c_1 \equiv c_2\ (\mathrm{mod}\: \hat{\mathfrak{p}}_{v_i}^{n_i}) \ \ \text{for all} \ \ i = 1, \ldots , t.
\end{equation}
On the other hand, for any $v \in V_f^k$ and any $u_1 , u_2 \in U_v$, the congruence $u_1 \equiv u_2\ (\mathrm{mod}\: \hat{\mathfrak{p}}_v^n)$ for $n \geq 1$ is equivalent to
$$
u_1 U_v^{(n)} = u_2 U_v^{(n)},
$$
where $U_v^{(n)}$ is the congruence subgroup of $U_v$ modulo $\hat{\mathfrak{p}}_v^n$. Thus, for (nonzero) $c_1 , c_2 \in \mathscr{O}$ prime to $\mathfrak{b}$, the conditions (\ref{E:X002}) and (\ref{E:X004}) are equivalent, and our assertion follows.
\end{proof}

We will now establish a result needed for the proof of Theorem \ref{dit} and its refinements.

\begin{prop}\label{dir}
Let $\mathfrak{b}$ be a nonzero ideal of $\mathscr{O}$,
let $a \in \mathscr{O}$ be relatively prime to $\mathfrak{b}$, and
let $F$ be a finite Galois extension of $\mathbb{Q}$ that contains $K(\mathfrak{b})$.
Assume that a rational prime $p$ is unramified in $F$ and there exists an extension $w$ of
the $p$-adic valuation $v_p$ to $F$ such that
$\mathrm{Fr}_{F/\Bbb Q}(w \vert v_p) \vert K(\mathfrak{b}) = \theta_{\mathfrak{b}}(a)$.
If the restriction $v$ of $w$ to $k$ does not belong to $S \cup V(\mathfrak{b})$ then:

\vskip2mm

\noindent {\rm (a)} $k_v = \mathbb{Q}_p$;

\vskip1mm

\noindent {\rm (b)} \parbox[t]{14.5cm}{the prime ideal $\mathfrak{p} = \mathfrak{p}_v$ of
$\mathscr{O}$ corresponding to $v$ is principal with a generator $\pi$ satisfying
$\pi \equiv a\ (\mathrm{mod}\: \mathfrak{b})$ and $\pi > 0$ in every real completion of $k$.}
\end{prop}
(We note since $v$ is unramified in $F$ which contains $K(\mathfrak{b})$, we in fact
{\it automatically} have that $v \notin V(\mathfrak{b})$.)
\begin{proof}
(a): Since the Frobenius $\mathrm{Fr}(w \vert v_p)$ generates $\mathrm{Gal}(F_w/\mathbb{Q}_p)$, our claim immediately follows from the fact that it acts trivially on $k$.

\vskip1mm

(b): According to (a), the local degree $[k_v:\Bbb Q_p]$ is 1, hence the residual degree
$f(v|v_p)$ is also 1, and therefore
$$
\mathrm{Fr}(w \vert v) = \mathrm{Fr}(w \vert v_p)^{f(v \vert v_p)} = \mathrm{Fr}(w \vert v_p).
$$
Thus,
$$
\alpha_{K(\mathfrak{b})/k}(\mathbf{i}(v)) = \mathrm{Fr}(w \vert v) \vert K(\mathfrak{b}) =
\theta_{\mathfrak{b}}(a) = \alpha_{K(\mathfrak{b})/k}(\mathbf{j}_{\mathfrak{b}}(a))^{-1},
$$
and therefore
$$
\mathbf{i}(v) \mathbf{j}_{\mathfrak{b}}(a) \in \mathrm{Ker}\: \alpha_{K(\mathfrak{b})/K} =
\mathbf{R}_S(\mathfrak{b}) = R_S(\mathfrak{b}) k^{\times}.
$$
So, we can write
\begin{equation}\label{E:X007}
\mathbf{i}(v) \mathbf{j}_{\mathfrak{b}}(a) = \mathbf{r} \pi \ \ \text{with} \ \ \mathbf{r} \in
R_S(\mathfrak{b}), \ \pi \in k^{\times}.
\end{equation}
Then
$$
\pi = \mathbf{i}(v) (\mathbf{j}_{\mathfrak{b}}(a) \mathbf{r}^{-1}).
$$
Since $a$ is prime to $\mathfrak{b}$, the idele $\mathbf{j}_{\mathfrak{b}}(a) \in U_S$, and then
$\mathbf{j}_{\mathfrak{b}}(a)\mathbf{r}^{-1}\in U_S$.  For any $v' \in V^k \setminus(S\cup\{v\})$,
the $v'$-component of $\mathbf{i}(v)$ is trivial, so we obtain that $\pi \in U_{v'}$.
On the other hand, the $v$-component of $\mathbf{i}(v)$ is a uniformizer $\pi_v$ of $k_v$ implying
that $\pi$ is also a uniformizer.  Thus, $\mathfrak{p} = \pi\mathscr{O}$ is precisely the prime
ideal associated with $v$.  For any real $v'$, the $v'$-components of $\mathbf{i}(v)$ and
$\mathbf{j}_{\mathfrak{b}}(a)$ are trivial, so $\pi$ equals the inverse of the $v'$-component of
$\mathbf{r}$, hence positive in $k_{v'}$.  Finally, it follows from (\ref{E:X007}) that
$$
\mathrm{pr}_{\mathfrak{b}}(\mathbf{j}_{\mathfrak{b}}(a)) = \mathrm{pr}_{\mathfrak{b}}(\mathbf{j}_{\mathfrak{b}}(\pi) \mathbf{r}),
$$
so $\pi\equiv\ a$ $(\mathrm{mod}\: \mathfrak{b})$ by Lemma \ref{dis}(b), as required.
\end{proof}

\renewcommand{\B}{\mathit{b}{\Cal O}}
$Proof$ $of$ $Theorem$ $\ref{dit}$.  Set $\mathfrak{b} = b\mathscr{O}$ and
$\sigma = \theta_{\mathfrak{b}}(a) \in \mathrm{Gal}(K(\mathfrak{b})/k)$.
Let $F$ be the Galois closure of $K(\mathfrak{b})$ over $\Bbb Q$, and
let $\tau \in \mathrm{Gal}(F/\mathbb{Q})$ be such that $\tau \vert K(\mathfrak{b}) = \sigma$.
Applying Chebotarev's Density Theorem (see [CF, Ch. VII, 2.4] or [BMS, A.6]), we find
infinitely many rational primes $p > 2$ for which the $p$-adic valuation $v_p$ is
unramified in $F$, does not lie below any valuations in $S \cup V(\mathfrak{b})$, and has an
extension $w$ to $F$ such that $\mathrm{Fr}_{F/\Bbb Q}(w \vert v_p) = \tau$.  Let $v = w \vert k$,
and let $\mathfrak{p} = \mathfrak{p}_v$ be the corresponding prime ideal of $\mathscr{O}$.
Since $p>2$, part (a) of Proposition \ref{dir} implies that $\mathfrak{p}$ is $\mathbb{Q}$-split.
Furthermore, part (b) of it asserts that $\mathfrak{p}$ has a generator $\pi$ such
that $\pi \equiv a$ $(\mathrm{mod}\: \mathfrak{b})$ and $\pi > 0$ in every real completion of $k$,
as required. \hfill $\Box$

\begin{rem}
Dong Quan Ngoc Nguyen pointed out to us that Theorem \ref{dit}, hence the essential part of Dirichlet's Theorem from \cite{BMS} (in particular, (A.11)), was known already to Hasse \cite[Satz 13]{Hasse}. In the current paper, however, we use the approach described in \cite{BMS} to establish the key Theorem \ref{T:Key}; the outline of the constructions from \cite{BMS} as well as the technical Lemma \ref{dis} and Proposition \ref{dir} are included for this purpose.
We note that in contrast to the argument in \cite{BMS}, our proofs of Theorems \ref{dit} and  \ref{T:Key} involve the application of Chebotarev's Density Theorem to {\it noncommutative} Galois extensions.
\end{rem}

We will now prove a statement from Galois theory that we will need in the next subsection.

\begin{lem}\label{L:Root}
Let $F/\mathbb{Q}$ be a finite Galois extension, and let $\kappa$ be an integer for which $F \cap \mathbb{Q}^{\mathrm{ab}} \subseteq
\mathbb{Q}(\zeta_{\kappa})$. Then $F(\zeta_{\kappa}) \cap \mathbb{Q}^{\mathrm{ab}} = \mathbb{Q}(\zeta_{\kappa})$.
\end{lem}
\begin{proof}
We need to show that
\begin{equation}\label{E:XXX1}
[F(\zeta_{\kappa}) : F(\zeta_{\kappa}) \cap \mathbb{Q}^{\mathrm{ab}}] = [F(\zeta_{\kappa}) : \mathbb{Q}(\zeta_{\kappa})].
\end{equation}
Let
$$
G = \mathrm{Gal}(F(\zeta_{\kappa})/\mathbb{Q}) \ \ \text{and} \ \ H = \mathrm{Gal}(F/\mathbb{Q}).
$$
Then the left-hand side of (\ref{E:XXX1}) is equal to the order of the commutator subgroup $[G , G]$, while the right-hand side equals
$$
[F : F \cap \mathbb{Q}(\zeta_{\kappa})] = [F : F \cap \mathbb{Q}^{\mathrm{ab}}] = \vert [H , H] \vert.
$$
Now, the restriction gives an {\it injective} group homomorphism
$$
\psi \colon G \to H \times \mathrm{Gal}(\mathbb{Q}(\zeta_{\kappa}) / \mathbb{Q}).
$$
Since the restriction $G \to H$ is surjective, we obtain that $\psi$ implements an isomorphism between $[G , G]$ and $[H , H] \times \{1\}$. Thus,
$[G , G]$ and $[H , H]$ have the same order, and (\ref{E:XXX1}) follows.

\end{proof}

\noindent {\bf 3. Key statement.} In this subsection we will establish another number-theoretic statement which plays a crucial role in the proof of Theorem 1.1. To formulate it, we need to introduce some additional notations. As above, let $\mu = \vert \mu(k)\vert$ be the number of roots of unity in $k$, let $K$ be the Hilbert $S$-class field of $k$, and let $\tilde{K}$ be the Galois closure of $K$ over $\mathbb{Q}$.  Suppose we are given two finite sets $P$ and $Q$ of rational primes.  Let
$$
\mu' = \mu \cdot \prod_{p \in P} p,
$$
pick an integer $\lambda \geq 1$ which is divisible by $\mu$ and for which $\tilde{K} \cap \mathbb{Q}^{\mathrm{ab}} \subseteq \mathbb{Q}(\zeta_{\lambda})$, and set
$$
\lambda' = \lambda \cdot \prod_{q \in Q} q.
$$

\begin{thm}\label{T:Key}
Let $u \in \mathscr{O}^{\times}$ be a unit of infinite order such that $u \notin \mu(k)_\r(k^{\times})^p$
for every prime $p \in P$, and let $\mathfrak{q}$ be a $\mathbb{Q}$-split prime of $\mathscr{O}$ which is
relatively prime to $\lambda'$.  Then there exist infinitely many principal $\mathbb{Q}$-split primes
$\mathfrak{p} = \pi \mathscr{O}$ of $\mathscr{O}$ with a generator $\pi$ such that

\vskip2mm

\noindent {\rm (1)} \parbox[t]{15cm}{for each $p \in P$, the $p$-primary component of $\phi(\mathfrak{p})/\mu$ divides the $p$-primary component of the order
of $u$ $(\mathrm{mod}\: \mathfrak{p})$;}

\vskip2mm

\noindent {\rm (2)} \parbox[t]{15cm}{$\pi (\mathrm{mod}\: \mathfrak{q}^2)$ generates $(\mathscr{O}/\mathfrak{q}^2)^{\times}$;}

\vskip2mm

\noindent {\rm (3)} \parbox[t]{15cm}{$\mathrm{gcd}(\phi(\mathfrak{p}) , \lambda') = \lambda$.}
\end{thm}
\begin{proof}
As in the proof of Theorem \ref{dit}, we will derive the required assertion by applying Chebotarev's Density Theorem to a specific automorphism of an appropriate finite Galois extension.

Let $K(\mathfrak{q}^2)$ be the abelian extension $K(\mathfrak{b})$ of $k$ introduced in subsection 3.2 for the ideal $\mathfrak{b} = \mathfrak{q}^2$. Set
$$
L_1 = K(\mathfrak{q}^2)(\zeta_{\lambda'}), \ \ \ L_2 = k\left(\zeta_{\mu'} , \sqrt[\mu']{u}\right), \ \ \ L = L_1 L_2 \ \ \ \text{and} \ \ \
\ell = L_1 \cap L_2.
$$
Then
\begin{equation}\label{E:Gal}
\mathrm{Gal}(L/k) = \{\, \sigma = (\sigma_1 , \sigma_2) \in \mathrm{Gal}(L_1/k) \times \mathrm{Gal}(L_2/k) \ \, \vert \ \, \sigma_1 \vert \ell = \sigma_2
\vert \ell \,\}.
\end{equation}
So, to construct $\sigma \in \mathrm{Gal}(L/k)$ that we will need in the argument it is enough to
construct appropriate $\sigma_i \in \mathrm{Gal}(L_i/k)$ for $i = 1, 2$ that have the same
restriction to $\ell$.

\begin{lem}\label{L:Isom}
The restriction maps define the following isomorphisms:

\vskip2mm

\noindent {\rm (1)} $\mathrm{Gal}(L_1/K) \simeq \mathrm{Gal}(K(\mathfrak{q}^2)/K) \times \mathrm{Gal}(K(\zeta_{\lambda'})/K)$;

\vskip2mm

\noindent {\rm (2)} $\displaystyle \mathrm{Gal}(K(\zeta_{\lambda'})/K(\zeta_{\lambda})) \simeq \mathrm{Gal}(\mathbb{Q}(\zeta_{\lambda'})/\mathbb{Q}(\zeta_{\lambda})) \simeq \prod_{q \in Q} \mathrm{Gal}(\mathbb{Q}(\zeta_{q\lambda})/\mathbb{Q}(\zeta_{\lambda}))$.
\end{lem}
\begin{proof}
(1): We need to show that $K(\mathfrak{q}^2) \cap K(\zeta_{\lambda}) = K$. But the Galois extensions $K(\mathfrak{q}^2)/K$ and $K(\zeta_{\lambda})/K$ are respectively totally and unramified at the extensions of $v_{\mathfrak{q}}$ to $K$ (since $\mathfrak{q}$ is prime to $\lambda$), so the required fact is immediate.

\vskip1mm

(2): Since $K(\zeta_{\lambda'}) = K(\zeta_{\lambda}) \cdot \mathbb{Q}(\zeta_{\lambda'})$, we only need to show that
\begin{equation}\label{E:Root}
K(\zeta_{\lambda}) \cap \mathbb{Q}(\zeta_{\lambda'}) = \mathbb{Q}(\zeta_{\lambda}).
\end{equation}
We have
$$
K(\zeta_{\lambda}) \cap \mathbb{Q}(\zeta_{\lambda'}) \subseteq \tilde{K}(\zeta_{\lambda}) \cap \mathbb{Q}^{\mathrm{ab}} = \mathbb{Q}(\zeta_{\lambda})
$$
by Lemma \ref{L:Root}. This proves one inclusion in (\ref{E:Root}); the other inclusion is obvious.
\end{proof}

Since $\mathfrak{q}$ is $\mathbb{Q}$-split, the group $(\mathscr{O}/\mathfrak{q}^2)^{\times}$ is cyclic (Lemma 3.1(a)), and we pick $c \in \mathscr{O}$ so that
$c$ $(\mathrm{mod}\: \mathfrak{q}^2)$ is a generator of this group. We then set
$$
\sigma'_1 = \theta_{\mathfrak{q}^2}(c) \in \mathrm{Gal}(K(\mathfrak{q}^2)/K)
$$
in the notations of subsection 3.2 (cf. Lemma \ref{dis}(a)).
Next, for $q \in Q$, we let $q^{e(q)}$ be the $q$-primary component of  $\lambda$.  Then using the
isomorphism from Lemma \ref{L:Isom}(2), we can find $\sigma''_1 \in \mathrm{Gal}(K(\zeta_{\lambda'})/K)$
such that
\begin{equation}\label{E:Action}
\sigma''_1(\zeta_{\lambda}) = \zeta_{\lambda} \ \ \text{but} \ \ \sigma''_1(\zeta_{q^{e(q) + 1}}) \neq \zeta_{q^{e(q) + 1}} \ \ \text{for all} \ \ q \in Q.
\end{equation}
We then define $\sigma_1 \in \mathrm{Gal}(L_1/K)$ to be the automorphism corresponding to the pair $(\sigma'_1 , \sigma''_1)$ in terms of the isomorphism from Lemma \ref{L:Isom}(1) (in other words, the restrictions of $\sigma_1$ to $K(\mathfrak{q}^2)$ and $K(\zeta_{\lambda'})$ are $\sigma'_1$ and $\sigma''_1$, respectively).

\vskip1mm

We fix a $\mu'$-th root $\sqrt[\mu']{u}$, and for $\nu \vert \mu'$ set $\sqrt[\nu]{u} =
\left( \sqrt[\mu']{u}  \right)^{\mu'/\nu}$ (also denoted $u^{\nu^{-1}}$).
To construct $\sigma_2 \in \mathrm{Gal}(L_2/k)$, we need the following.
\begin{lem}\label{L:sigma2}
Let 
$\sigma_0 \in \mathrm{Gal}(\ell/k)$.
Then there exists $\sigma_2 \in \mathrm{Gal}(L_2/k)$ such that

\vskip2mm

\noindent {\rm (1)} $\sigma_2 \vert \ell = \sigma_0$;

\vskip2mm

\noindent {\rm (2)} \parbox[t]{15cm}{for any $p \in P$, if $p^{d(p)}$ is the $p$-primary component
of
$\mu$ then $$\sigma_2\left( u^{p^{-(d(p) +1)}}\right) \neq u^{p^{-(d(p) +1)}},$$ and consequently
either $\sigma_2(\zeta_{p^{d(p)+1}}) \neq \zeta_{p^{d(p)+1}}$ or $\sigma_2$ acts nontrivially on
\hskip.5pt \underline{all} \hskip.8pt $p^{d(p)+1}$-th roots of $u$.}
\end{lem}

\begin{proof}
Since $L_1/k$ is an abelian extension, we conclude from Corollary \ref{L:max-ab} that
\begin{equation}\label{E:incl}
\ell \subseteq k\left(\sqrt[\mu]{u} , \zeta_{\mu'}\right)\subseteq k^{\rm ab}.
\end{equation}
On the other hand, according to Proposition \ref{P:PPP2}, none of the roots $\sqrt[p\mu]{u}$ for
$p \in P$ lies in $k^{\rm ab}$, and the restriction maps yield an isomorphism $$
\mathrm{Gal}\left(k\big(\sqrt[\mu']{u},\zeta_{\mu'}\big)/k\left(\sqrt[\mu]{u},\zeta_{\mu'}\right)\right)
\rightarrow\prod_{p\in P}\mathrm{Gal}\Big(k\left(\sqrt[p\mu]{u},\zeta_{\mu'}\right)/
k\left(\sqrt[\mu]{u},\zeta_{\mu'}\right)\Big).
$$
It follows that for each $p\in P$ we can find $\tau_p \in \mathrm{Gal}\Big(k\left(\sqrt[\mu']{u} , \zeta_{\mu'}\right)/k\left(\sqrt[\mu]{u} , \zeta_{\mu'}\right)\Big)$ such that
$$
\tau_p\left(u^{p^{-(d(p) +1)}}\right) = \zeta_p \cdot u^{p^{-(d(p) +1)}} \ \ \ \text{and} \ \ \ \tau_p\left(u^{q^{-(d(q) +1)}}\right) = u^{q^{-(d(q) +1)}} \ \ \ \text{for all} \ \ \ q \in P \setminus \{p\}.
$$
Now, let $\tilde{\sigma}_0$ be any extension of $\sigma_0$ to $L_2$. For $p \in P$, define
$$
\chi(p) = \left\{ \begin{array}{ccl} 1 & , & \tilde{\sigma}_0\left(u^{p^{-(d(p) +1)}}\right) = u^{p^{-(d(p) +1)}} \\
0 & , & \tilde{\sigma}_0\left(u^{p^{-(d(p) +1)}}\right) \neq u^{p^{-(d(p) +1)}}
\end{array} \right.
$$
Set
$$
\sigma_2 = \tilde{\sigma}_0\cdot\prod_{p \in P} \tau_p^{\chi(p)}.
$$
In view of (\ref{E:incl}), all $\tau_p$'s act trivially on $\ell$, so
$\sigma_2 \vert \ell = \tilde{\sigma}_0 \vert \ell = \sigma_0$ and (1) holds.
Furthermore, the choice of the $\tau_p$'s and the $\chi(p)$'s implies that (2) also holds.
\end{proof}

\vskip1mm

Continuing the proof of Theorem \ref{T:Key}, we now use $\sigma_1 \in \mathrm{Gal}(L_1/k)$ constructed above, set $\sigma_0 = \sigma_1 \vert \ell$, and using Lemma \ref{L:sigma2} construct $\sigma_2 \in \mathrm{Gal}(L_2/k)$ with the properties described therein. In particular, part (1) of this lemma in conjunction with (\ref{E:Gal}) implies that the pair $(\sigma_1 , \sigma_2)$ corresponds to an automorphism $\sigma \in \mathrm{Gal}(L/k)$. As in the proof of Theorem \ref{dit}, we let $F$ denote the Galois closure of $L$ over $\mathbb{Q}$, and let $\tilde{\sigma} \in \mathrm{Gal}(F/\mathrm{Q})$ be such that $\tilde{\sigma} \vert L = \sigma$. By Chebotarev's Density Theorem, there exist infinitely many rational primes $\pi > 2$ that are relatively prime to $\lambda' \cdot \mu'$ and for which the $\pi$-adic valuation $v_{\pi}$ is unramified in $F$, does not lie below any valuation in $S \cup \{ v_{\mathfrak{q}} \}$, and has an extension $w$ to $F$ such that $\mathrm{Fr}_{F/\Bbb Q}(w \vert v_{\pi}) = \tilde{\sigma}$.  Let $v=w \vert k$, and let $\mathfrak{p} = \mathfrak{p}_v$ be the corresponding prime ideal of $\mathscr{O}$. As in the proof of Theorem \ref{dit}, we see that $\mathfrak{p}$ is $\mathbb{Q}$-split. Furthermore, since $\sigma \vert K(\mathfrak{q}^2) = \theta_{\mathfrak{q}^2}(c)$, we conclude that $\mathfrak{p}$ has a generator $\pi$ such that $\pi \equiv c$ $(\mathrm{mod}\: \mathfrak{q}^2)$ (cf. Proposition \ref{dir}(b)). Then by construction $\pi$ $(\mathrm{mod}\: \mathfrak{q}^2)$ generates $(\mathscr{O}/\mathfrak{q}^2)^\times$, verifying condition (2) of Theorem \ref{T:Key}.

\vskip1mm

To verify condition (1), we fix $p \in P$ and consider two cases. First, suppose $\sigma(\zeta_{p^{d(p) + 1}}) \neq \zeta_{p^{d(p) + 1}}$. Since $p$ is prime to $\mathfrak{p}$, this means that the residue field $\mathscr{O}/\mathfrak{p}$ does not contain an element of order $p^{d(p) + 1}$ (although, since $\mu$ is prime to $\mathfrak{p}$, it does contain an element of order $\mu$, hence of order $p^{d(p)}$).  So, in this case $\phi(\mathfrak{p})/\mu$ is prime to $p$, and there is nothing to prove. Now, suppose that  $\sigma(\zeta_{p^{d(p) + 1}}) = \zeta_{p^{d(p) + 1}}$. Then by construction $\sigma$ acts nontrivially on every $p^{d(p)+1}$-th root of $u$, and therefore the polynomial $X^{p^{d(p)+1}} - u$ has no roots in $k_{v}$. Again, since $p$ is prime to $\mathfrak{p}$, we see from Hensel's lemma that $u$ $(\mathrm{mod}\: \mathfrak{p})$ is not a $p^{d(p)+1}$-th power in the residue field. It follows that the $p$-primary component of the order of $u$ $(\mathrm{mod}\: \mathfrak{p})$ is not less than the $p$-primary component of $\phi(\mathfrak{p})/p^{d(p)}$, and (1) follows.

\vskip1mm

Finally, by construction $\sigma$ acts trivially on $\zeta_{\lambda}$ but nontrivially on $\zeta_{q\lambda}$ for any $q \in Q$. Since $\mathfrak{p}$ is prime to $\lambda'$, we see that the residue field $\mathscr{O}/\mathfrak{p}$ contains an element of order $\lambda$, but does not contain an element of order $q\lambda$ for any $q \in Q$. This means that $\lambda$ $\vert$ $\phi(\mathfrak{p})$ but $\phi(\mathfrak{p})/\lambda$ is relatively prime to each $q \in Q$, which is equivalent to condition (3) of Theorem \ref{T:Key}.
\end{proof}

\section{Proof of Theorem 1.1}
\newcommand{\M}{m}
First, we will introduce some additional notations needed to convert the task of factoring
a given matrix $A\in\rm SL_2(\Cal O)$ as a product of elementary matrices into the task of reducing
the first row of $A$ to $(1,0)$.  Let
$$
\mathscr{R}(\mathscr{O}) = \{ (a , b) \in \mathscr{O}^2 \: \vert \: a\mathscr{O} + b\mathscr{O} =
\mathscr{O} \}
$$
(note that $\mathscr{R}(\mathscr{O})$ is precisely the set of all first rows of matrices $A \in \mathrm{SL}_2(\mathscr{O})$). For $\lambda \in \mathscr{O}$, one defines two permutations, $e_{+}(\lambda)$ and $e_{-}(\lambda)$, of $\mathscr{R}(\mathscr{O})$ given respectively by
$$
(a , b) \mapsto (a , b + \lambda a) \ \ \text{and} \ \ (a , b) \mapsto (a + \lambda b , b).
$$
These permutations will be called {\it elementary transformations} of $\mathscr{R}(\mathscr{O})$. For $(a , b)$, $(c , d) \in \mathscr{R}(\mathscr{O})$ we will write $(a , b) \stackrel{n}{\Rightarrow} (c , d)$ to indicate the fact that $(c , d)$ can be obtained from $(a , b)$ by a sequence of $n$ (equivalently, $\leq n$) elementary transformations. For the convenience of further reference, we will record some simple properties of this relation.

\begin{lem}\label{elem}
Let $(a,b)\in\Cal R(\Cal O)$.

\noindent {\rm (1a)} If $(c,d)\in\Cal R(\Cal O)$ and $(a , b) \stackrel{n}{\Rightarrow} (c , d)$,
then $(c , d) \stackrel{n}{\Rightarrow} (a , b)$.

\vskip1mm

\noindent {\rm (1b)} If $(c,d),(e,f)\in\Cal R(\Cal O)$ are such that
$(a , b) \stackrel{m}{\Rightarrow} (c , d)$ and
$(c , d) \stackrel{n}{\Rightarrow} (e , f)$, then $(a , b) \stackrel{m+n}{\Rightarrow} (e , f)$.

\vskip1mm

\noindent {\rm (2a)} If $c\in\Cal O$ such that $c \equiv a(\mathrm{mod}\: b\mathscr{O})$, then
$(c,b)\in\Cal R(\Cal O)$, and $(a , b) \stackrel{1}{\Rightarrow} (c , b)$.

\vskip1mm

\noindent {\rm (2b)} If $d\in\Cal O$ such that $d \equiv b(\mathrm{mod}\: a\mathscr{O})$, then
$(a,d)\in\Cal R(\Cal O)$, and $(a , b) \stackrel{1}{\Rightarrow} (a , d)$.

\vskip1mm

\noindent {\rm (3a)} If $(a , b) \stackrel{n}{\Rightarrow} (1 , 0)$ then any matrix
$A \in \mathrm{SL}_2(\mathscr{O})$ with the first row $(a , b)$ is a product of
$\leq n + 1$ elementary matrices.

\vskip1mm

\noindent {\rm (3b)} If $(a , b) \stackrel{n}{\Rightarrow} (0 , 1)$ then any matrix
$A\in\mathrm{SL}_2(\Cal O)$ with the second row $(a , b)$ is a product of $\leq n+1$ elementary matrices.

\noindent {\rm (4a)} If $a\in\Cal O^\times$ then $(a,b)\stackrel{2}{\Rightarrow}(0,1)$.

\noindent {\rm (4b)} If $b\in\Cal O^\times$ then $(a,b)\stackrel{2}{\Rightarrow}(1,0)$.
\end{lem} \begin{proof}
For (1a), we observe that the inverse of an elementary transformation is again an elementary
transformation given by $[e_{\pm}(\lambda)]^{-1} = e_{\pm}(-\lambda)$, so the required fact
follows.  Part (1b) is obvious.

(Note that (1) implies that the relation between $(a,b)$ and $(c,d)\in\Cal R(O)$
defined by $(a , b) \stackrel{n}{\Rightarrow} (c , d)$ for \emph{some} $n \in \mathbb{N}$
is an equivalence relation.)

\vskip2mm

In (2a), we have $c = a + \lambda b$ with $\lambda \in \mathscr{O}$.  Then $$
c{\Cal O}+b{\Cal O}=a{\Cal O}+b{\Cal O}={\Cal O},
$$
so $(c,a)\in\Cal R({\Cal O})$, and $e_{+}(\lambda)$ takes $(a , b)$ to $(c , b)$.
The argument for (2b) is similar.

\vskip2mm

(3a) Suppose $A \in \mathrm{SL}_2(\mathscr{O})$ has the first row $(a , b)$. Then for
$\lambda \in \mathscr{O}$, the first row of the product $AE_{12}(\lambda)$ is
$(a , b + \lambda a) = e_{+}(\lambda)(a , b)$, and similarly the first row of
$AE_{21}(\lambda)$ is $e_{-}(\lambda)(a , b)$.  So, the fact that
$(a , b) \stackrel{n}{\Rightarrow} (1 , 0)$ implies that there exists a matrix
$U \in \mathrm{SL}_2(\mathscr{O})$ which is a product of $n$ elementary matrices and
is such that $AU$ has the first row $(1 , 0)$. This means that $AU = E_{21}(z)$
for some $z\in \mathscr{O}$, and then $A = E_{21}(z)U^{-1}$ is a product of
$\leq n + 1$ elementary matrices.  The argument for (3b) is similar.

Part (4a) follows since $e_-\big(-$$a\big)e_+\big(a^{-1}(1-b)\big)(a,b)=(0,1)$.  The proof of
(4b) is similar.
\end{proof}

\noindent {\it Remark.} All assertions of Lemma \ref{elem} are valid over any commutative ring $\mathscr{O}$.

\begin{cor}\label{s3} Let $\mathfrak{q}$ be a principal $\Bbb Q$-split prime ideal of $\Cal O$
with generator $q$, and let $z \in \mathscr{O}$ be such that $z(\mathrm{mod}\: \frak q^2)$
generates $(\mathscr{O}/\mathfrak{q}^2)^{\times}$. Given an element of $\Cal R(\Cal O)$ of the
form $(b,q^n)$ with $n\geq 2$, and an integer $t_0$, there exists an integer
$t \geq t_0$ such that $(b,q^n) \stackrel{1}{\Rightarrow} (z^t,q^n)$.
\end{cor}\begin{proof}
By Lemma \ref{prim}(b), the element $z(\mathrm{mod}\: \mathfrak{q}^n)$
generates $(\mathscr{O}/\mathfrak{q}^n)^{\times}$. Since $b$ is prime to $\mathfrak{q}$,
one can find $t \in \Bbb Z$ such that $b \equiv z^t(\mathrm{mod}\: \mathfrak{q}^n)$. Adding
to $t$ a suitable multiple of $\phi(\mathfrak{q}^n)$ if necessary, we can assume that
$t \geq t_0$. Our assertion then follows from Lemma \ref{elem}(2a).
\end{proof}

\renewcommand{\q}{d}
\begin{lem}\label{s0}
Suppose we are given $(a,b)\in\Cal R(\Cal O)$, a finite subset $T\subseteq V^k_f$, and an integer
$n\neq 0$.  Then there exists $\alpha\in\Cal O$$_k$ and $r\in\Cal O^\times$ such that
$V(\alpha)\cap T=\emptyset$, and $(a,b)\stackrel{1}{\Rightarrow}(\alpha r^n,b)$.
\end{lem}
\begin{proof}
Let $h_k$ be the class number of $k$.  If for each $v \in S\setminus V^k_\infty$ we let
$\frak m$$_v$ denote the maximal ideal of $\Cal O$$_k$ corresponding to $v$, then the ideal
$(\frak m_v)^{h_k}$ is principal, and its generator $\pi_v$ satisfies $v(\pi_v)=h_k$ and
$w(\pi_v)=0$ for all $w\in V_f^k\setminus\{v\}$.  Let $R$ be the subgroup of $k^\times$
generated by $\pi_v$ for $v \in S\setminus V^k_\infty$; note that $R\subset\Cal O^\times$.
We can pick $r\in R$ so that $a' := ar^{-n} \in \mathscr{O}_k$.  We note that since
$a$ and $b$ are relatively prime in $\Cal O$, we have $V(a') \cap V(b) \subset S$.

Now, it follows from the strong approximation theorem that
there exists $\gamma \in\mathscr{O}_k$ such that $$
\begin{array}{rl}v(\gamma b)\geq 0{\rm\ and\ }v(\gamma b)\equiv 0(\mathrm{mod}\ nh_k)&
{\rm for\ all\ }v\in S\setminus V^k_\infty,\\
{\rm and\ }v(\gamma b) = 0&{\rm for\ all\ }v \in V(a') \setminus S.
\end{array}
$$
Then, in particular, we can find
$s \in R$ so that $v(\gamma bs^{-1}) = 0$ for all $v \in S\setminus V^k_\infty$.  Set $$
\gamma' := \gamma s^{-1} \in \mathscr{O}{\rm\ and\ }b' := \gamma'b\in\mathscr{O}_k.
$$
By construction,
\begin{equation}\label{E:777}
v(b') = 0 \ \ \text{for all} \ \ v \in V(a') \cup(S\setminus V^k_\infty),
\end{equation}
implying that $V(a')\cap V(b')=\emptyset$, which means that $a'$ and $b'$ are relatively prime in
$\mathscr{O}_k$.

Again, by the strong approximation theorem we can find $t \in \mathscr{O}_k$ such that $$
v(t) = 0{\rm\ for\ }v \in T \cap V(a'){\rm\ and\ }v(t) > 0{\rm\ for\ }v \in T \setminus V(a').
$$
Set $\alpha=a'+tb'\in\Cal O$$_k$.  Then for $v\in T\cap V(a')$ we have $v(a') > 0$ and $v(tb')=0$
(in view of (\ref{E:777})), while for $v\in T\setminus V(a')$ we have $v(a')=0$ and $v(tb')>0$.
In either case, $$
v(\alpha)=v(a' + tb') = 0{\rm\ for\ all\ }v\in T,
$$
i.e.~$V(\alpha)\cap T=\emptyset$.  On the other hand, $$
a + r^{n}t\gamma' b = r^{n} (a' + tb') = r^{n} \alpha,
$$
which means that $(a,b)\stackrel{1}{\rightarrow}(\alpha r^n,b)$, as required.

\end{proof}

Recall that we let $\mu$ denote the number of roots of unity in $k$.

\renewcommand{\P}{\frak m}
\renewcommand{\q}{q}
\newcommand{\m}{n}
\begin{lem}\label{s1}
Let $(a , b) \in \mathscr{R}(\mathscr{O})$ be such that $a = \alpha \cdot r^{\mu}$ for some $\alpha \in \mathcal{O}_k$ and $r \in \mathscr{O}^{\times}$ where $V(\alpha)$ is disjoint from $S \cup V(\mu)$. Then there exist $a' \in \mathscr{O}$ and infinitely many $\Bbb Q$-split prime principal ideals $\mathfrak{q}$ of $\mathscr{O}$ with a generator $q$ such that for any $m \equiv 1(\mathrm{mod}\: \phi(a'\mathscr{O}))$ we have $(a , b) \stackrel{3}{\Rightarrow} (a' , q^{\mu m})$.
\end{lem}
\begin{proof}
The argument below is adapted from the proof of Lemma 3 in [CK1]. It relies on the properties of the power residue symbol (in particular, the power reciprocity law) described in the Appendix on Number Theory in [BMS]. We will work with all $v \in V^k$ (and not only $v \in V^k \setminus S$), so to each such $v$ we associate a symbol (``modulus") $\mathfrak{m}_v$. For $v \in V^k_f$ we will identify $\mathfrak{m}_v$ with the corresponding maximal ideal of $\mathscr{O}_k$ (obviously, $\mathfrak{p}_v = \mathfrak{m}_v \mathscr{O}$ for $v \in V^k \setminus S$); the valuation ideal  and the group of units in the valuation ring $\mathscr{O}_v$ (or $\mathscr{O}_{\mathfrak{m}_v}$) in the completion $k_v$ will be denoted $\hat{\mathfrak{m}}_v$ and $U_v$ respectively. For any divisor $\kappa \vert \mu$, we let
$$
\left( \frac{* , *}{\ \mathfrak{m}_v} \right)_{\!\!\kappa}
$$
be the (bi-multiplicative, skew-symmetric) power residue symbol of degree $\kappa$ on $k_v^{\times}$ (cf. [BMS, p. 85]). We recall that $\displaystyle \left( \frac{x , y}{\ \mathfrak{m}_v} \right)_{\!\!\kappa}\!=1$ if one of the elements $x , y$ is a $\kappa$-th power in $k_v^{\times}$ (in particular, if either $v$ is complex or $v$ is real and one of the elements $x , y$ is positive in $k_v$) or if $v$ is nonarchimedean $\notin V(\kappa)$ and $x , y \in U_v$. It follows that for any $x , y \in k^{\times}$, we have $\displaystyle \left( \frac{x , y}{\ \mathfrak{m}_v}  \right)_{\!\!\kappa} = 1$ for almost all $v \in V^k$. Furthermore, we have the {\it reciprocity law}:
\begin{equation}\label{E:Recipr}
\prod_{v \in V^k} \left( \frac{x , y}{\ \mathfrak{m}_v}   \right)_{\!\!\kappa} = 1.
\end{equation}

Now, let $\mu = p_1^{e_1} \cdots p_n^{e_n}$ be a prime factorization of $\mu$. For each $i = 1, \ldots , n$, pick $v_i \in V(p_i)$. According to [BMS, (A.17)], the values
$$
\left( \frac{x , y}{\ \mathfrak{m}_{v_i}}   \right)_{\!p_i^{e_i}} \ \ \text{for} \ \ x , y \in U_{v_i}
$$
cover all $p_i^{e_i}$-th roots of unity. Thus, we can pick units $u_i , u'_i \in U_{v_i}$ for $i = 1, \ldots , n$ so that $\left( \frac{u_i , u'_i}{\ \mathfrak{m}_{v_i}} \right)_{\!p_i^{e_i}}\!\!$ is a primitive $p_i^{e_i}$-th root of unity. On the other hand, since $u_i , u'_i \in U_{v_i}$ and $v_i(\mu/p_i^{e_i}) = 0$, we have
$$
\left( \frac{u_i , u'_i}{\ \mathfrak{m}_{v_i}} \right)^{p_i^{e_i}}_{\!\mu}\!\! = \left( \frac{u_i , u'_i}{\ \mathfrak{m}_{v_i}} \right)_{\!\mu/p_i^{e_i}}\!\! = 1.
$$
Thus,
$$
\zeta_{p_i^{e_i}} := \left( \frac{u_i , u'_i}{\ \mathfrak{m}_{v_i}} \right)_{\!\mu}\!\!
$$
is a primitive $p_i^{e_i}$-th root of unity for each $i = 1, \ldots , n$ , making
\begin{equation}\label{E:zeta}
\zeta_{\mu} := \prod_{i = 1}^n \left( \frac{u_i , u'_i}{\ \mathfrak{m}_{v_i}} \right)_{\!\mu}
\end{equation}
a primitive $\mu$-th root of unity. Furthermore, it follows from the Inverse Function Theorem or
Hensel's Lemma that we can find an integer $N > 0$ such that
\begin{equation}\label{E:mu1}
1 + \hat{\mathfrak{m}}_v^N \subset {k_v^{\times}}^{\mu} \ \ \text{for all} \ \ v \in V(\mu).
\end{equation}

We now write $b = \beta t^{\mu}$ with $\beta \in \mathscr{O}_k$ and $t \in \mathscr{O}^{\times}$. Since $a , b$ are relatively prime in $\mathscr{O}$, so are $\alpha , \beta$, hence $V(\alpha) \cap V(\beta) \subset S$. On the other hand, by our assumption $V(\alpha)$ is disjoint from $S \cup V(\mu)$, so we conclude that $V(\alpha)$ is disjoint from $V(\beta) \cup V(\mu)$. Applying Theorem 3.3 to the ring $\mathscr{O}_k$ we obtain that there exists $\beta' \in \mathscr{O}_k$ having the following properties:

\vskip2mm

$(1)_1$ $\mathfrak{b} := \beta'\mathscr{O}_k$ is a prime ideal of $\mathscr{O}_k$ and the corresponding valuation $v_{\mathfrak{b}} \notin S \cup V(\mu)$;

\vskip1mm

$(2)_1$ $\beta' > 0$ in every real completion of $k$;

\vskip1mm

$(3)_1$ $\beta' \equiv \beta (\mathrm{mod}\: \alpha \mathscr{O}_k)$;

\vskip1mm

$(4)_1$ for each $i = 1, \ldots , n$, we have

\vskip1mm

\hskip10mm \parbox[t]{10cm}{$\beta' \equiv u'_i (\mathrm{mod}\: \hat{\frak m}^N_{v_i})$, and \vskip1mm $\beta' \equiv 1(\mathrm{mod}\: \hat{\frak m}^N_v)$ for all $v \in V(p_i) \setminus \{ v_i \}$.}

\vskip2mm

\noindent Set $b' = \beta' t^{\mu}$. It is a consequence of $(3)_1$ that $b \equiv b'(\mathrm{mod}\: a\mathscr{O})$, so by Lemma 4.1(2) we have $(a , b) \stackrel{1}{\Rightarrow} (a , b')$. Furthermore, it follows from $(4)_1$ and (\ref{E:mu1}) that $\beta'/u'_i \in {k^{\times}_{v_i}}^{\mu}$, so
$$
\left( \frac{u_i , \beta'}{\ \mathfrak{m}_{v_i}}  \right)_{\!\mu} \!\!= \left( \frac{u_i , u'_i}{\ \mathfrak{m}_{v_i}}  \right)_{\mu} \!\!= \zeta_{\!p_i^{e_i}}.
$$
Since $\zeta_{\mu}$ defined by (\ref{E:zeta}) is a primitive $\mu$-th root of unity, we can find an integer $d > 0$ such that
\begin{equation}\label{E:ZZZ1}
1 = \left( \frac{\alpha , \beta'}{\ \mathfrak{b}}  \right)_{\!\mu}\!\! \cdot \zeta_{\mu}^d = \left( \frac{\alpha , \beta'}{\ \mathfrak{b}}  \right)_{\!\mu}\!\cdot\ \prod_{i = 1}^n \left( \frac{u_i^d , \beta'}{\ \mathfrak{m}_{v_i}} \right)_{\!\mu} .
\end{equation}

By construction, $v_{\mathfrak{b}} \notin V(\alpha) \cup V(\mu)$, so applying Theorem 3.3 one more time, we find $\alpha' \in \mathscr{O}_k$ such that

\vskip2mm

$(1)_2$ $\mathfrak{a} := \alpha' \mathscr{O}_k$ is a prime ideal of $\mathscr{O}_k$ and the corresponding valuation $v_{\mathfrak{a}} \notin S \cup V(\mu)$;

\vskip1mm

$(2)_2$ $\alpha' \equiv \alpha (\mathrm{mod}\: \mathfrak{b})$;

\vskip1mm

$(3)_2$ $\alpha' \equiv u_i^d (\mathrm{mod}\: \hat{\mathfrak{m}}_{v_i}^N)$ for $i = 1, \ldots , n$.

\vskip2mm

\noindent Set $a' = \alpha' r^{\mu}$. Then $a'\mathscr{O} = \alpha' \mathscr{O}$ is a prime ideal of $\mathscr{O}$ and $a' \equiv a (\mathrm{mod}\: b' \mathscr{O})$, so $(a , b') \stackrel{1}{\Rightarrow} (a' , b')$.

Now, we note that  $\displaystyle \left( \frac{\alpha' , \beta'}{\ \mathfrak{m}_v}  \right)_{\!\mu}\!\! = 1$ if either $v \in V^k_{\infty}$ (since $\beta' > 0$ in all real completions of $k$) or $v \in V^k_f \setminus (V(\alpha') \cup V(\beta') \cup V(\mu))$. Since the ideals $\mathfrak{a} = \alpha' \mathscr{O}_k$ and $\mathfrak{b} = \beta' \mathscr{O}_k$ are prime by construction, we have $V(\alpha') = \{ v_{\mathfrak{a}} \}$ and $V(\beta') = \{ v_{\mathfrak{b}} \}$. Besides, it follows from (\ref{E:mu1}) and $(4)_1$ that for $v \in V(p_i) \setminus \{ v_i \}$ we have $\beta' \in {k^{\times}_v}^{\mu}$, and therefore again $\displaystyle \left( \frac{\alpha' , \beta'}{\ \mathfrak{m}_v}  \right)_{\!\mu}\!\! = 1$. Thus, the reciprocity law (\ref{E:Recipr}) for $\alpha' , \beta'$ reduces to the relation
\begin{equation}\label{E:Recipr2}
\left( \frac{\alpha' , \beta'}{\ \mathfrak{a}}  \right)_{\!\mu}\!\! \cdot \left( \frac{\alpha' , \beta'}{\ \mathfrak{b}}  \right)_{\!\mu}\!\cdot\ \prod_{i = 1}^n \left( \frac{\alpha' , \beta'}{\ \mathfrak{m}_{v_i}}  \right)_{\!\mu}\!\! = 1.
\end{equation}
It follows from $(2)_2$ and $(3)_2$ that
$$
 \left( \frac{\alpha' , \beta'}{\ \mathfrak{b}}  \right)_{\!\mu}\!\! =  \left( \frac{\alpha , \beta'}{\ \mathfrak{b}}  \right)_{\!\mu} \ \ \text{and} \ \
  \left( \frac{\alpha' , \beta'}{\ \mathfrak{m}_{v_i}}  \right)_{\!\mu}\!\! =  \left( \frac{u_i^d , \beta'}{\ \mathfrak{m}_{v_i}}  \right)_{\!\mu} \ \ \text{for all} \ \ i = 1, \ldots , n.
$$
Comparing now (\ref{E:ZZZ1}) with (\ref{E:Recipr2}), we find that
$$
\left( \frac{\beta' , \alpha'}{\ \mathfrak{a}} \right)_{\!\mu}\!\! = \left( \frac{\alpha' , \beta'}{\ \mathfrak{a}} \right)_{\!\mu}^{\!\!-1}\!\! = 1.
$$
This implies (cf. [BMS, (A.16)]) that $\beta'$ is a $\mu$-th power modulo $\mathfrak{a}$, i.e. $\beta' \equiv \gamma^{\mu} (\mathrm{mod}\: \mathfrak{a})$ for some $\gamma \in \mathscr{O}_k$. Clearly, the elements $a' = \alpha' r^{\mu}$ and $\gamma t$ are relatively prime in $\mathscr{O}$, so applying Theorem 3.3 to this ring, we find infinitely many $\Bbb Q$-split principal prime ideals $\mathfrak{q}$ of $\mathscr{O}$ having a generator $q \equiv \gamma t (\mathrm{mod}\: a'\mathscr{O})$. Then for any $m \equiv 1(\mathrm{mod}\: \phi(a'\mathscr{O}))$ we have
$$
q^{\mu m} \equiv q^{\mu} \equiv \beta' t^\mu \equiv b' (\mathrm{mod}\: a'\mathscr{O}),
$$
so $(a' , b') \stackrel{1}{\Rightarrow} (a' , q^{\mu m})$. Then by Lemma 4.1(1b), we have $(a , b) \stackrel{3}{\Rightarrow} (a' , q^{\mu m})$, as required.
\end{proof}

The final ingredient that we need for the proof of Theorem 1.1 is the following lemma which uses the notion of the {\it level} $\ell_{\mathfrak{p}}(u)$ of a unit $u$ of infinite order with respect to a $\Bbb Q$-split ideal $\mathfrak{p}$  introduced in \S3.1.
\begin{lem}\label{s4}
Let $\mathfrak{p}$ be a principal $\Bbb Q$-split ideal of $\mathscr{O}$ with a generator $\pi$, and let $u \in \mathscr{O}^{\times}$ be a unit of infinite order. Set $s = \ell_{\mathfrak{p}}(u)$, and let $\lambda$ and $m$ be integers satisfying $\lambda \vert \phi(\mathfrak{p})$ and $m \equiv 0(\mathrm{mod}\: \phi(\mathfrak{p}^s)/\lambda)$. Given an integer $\delta > 0$ dividing $\lambda$ and $b \in \mathscr{O}$ prime to $\pi$ such that $b$ is a $\delta$-th power ($\mathrm{mod}\: \mathfrak{p}$) while $\nu := \lambda/\delta$ divides the order of $u(\mathrm{mod}\: \mathfrak{p})$, for any integer $t \geq s$ there exists an integer $n_t$ for which $$(\pi^t , b^m) \stackrel{1}{\Rightarrow} (\pi^t , u^{n_t}).$$
\end{lem}
\begin{proof}
Let $p$ be the rational prime corresponding to $\mathfrak{p}$. Being a divisor of $\lambda$, the integer $\delta$ is relatively prime to $p$. So, the fact that $b$ is a $\delta$-th power $\mathrm{mod}\: \mathfrak{p}$ implies that it is also a $\delta$-th power $\mathrm{mod}\: \mathfrak{p}^s$. On the other hand,  it follows from our assumptions that $\lambda m = \delta \nu m$ is divisible by $\phi(\mathfrak{p}^s)$, and therefore $(b^m)^{\nu} \equiv 1(\mathrm{mod}\: \mathfrak{p}^s)$. But since $\nu$ is prime to $p$, the subgroup of elements in $(\mathscr{O}/\mathfrak{p}^s)^{\times}$ of order dividing $\nu$ is isomorphic to a subgroup of $(\mathscr{O}/\mathfrak{p})^{\times}$, hence cyclic. So, the fact that the order of $u(\mathrm{mod}\: \mathfrak{p})$, and consequently the order $u(\mathrm{mod}\: \mathfrak{p}^s)$, is divisible by $\nu$ implies that every element in $(\mathscr{O}/\mathfrak{p}^s)^{\times}$ whose order divides $\nu$ lies in the subgroup generated by $u(\mathrm{mod}\: \mathfrak{p}^s)$. Thus, $b^m \equiv u^{n_s} (\mathrm{mod}\: \mathfrak{p}^s)$ for some integer $n_s$. Since $\mathfrak{p}$ is $\Bbb Q$-split, we can apply Lemma 3.2 to conclude  that for any $t \geq s$ there exists an integer $n_t$ such that $b^m \equiv u^{n_t}(\mathrm{mod}\: \mathfrak{p}^t)$. Then $(\pi^t , b^m) \stackrel{1}{\Rightarrow} (\pi^t , u^{n_t})$ by Lemma 4.1(2).
\end{proof}
\renewcommand{\q}{q}

\vskip3mm

We will call a unit $u \in \mathscr{O}^{\times}$ {\it fundamental} if it has infinite order and the cyclic group $\langle u \rangle$ is a direct factor of
$\mathscr{O}^{\times}$. Since the group $\mathscr{O}^{\times}$ is finitely generated (Dirichlet's Unit Theorem, cf. [CF, \S2.18]) it always contains a fundamental unit once it is infinite. We note that any fundamental unit has the following property:
$$
u \notin \mu(k)_\r(k^{\times})^p \ \ \text{for any prime} \  \ p.
$$
We are now in a position to give

\vskip1.5mm

\begin{proof}[Proof of Theorem 1.1.] We return to the notations of \S 3.3: we let $K$ denote the Hilbert $S$-class field of $k$, let $\tilde{K}$ be its normal closure over $\Bbb Q$, and pick an integer $\lambda \geq 1$ which is divisible by $\mu$ and for which $\tilde{K} \cap \Bbb Q^{\mathrm{ab}} \subset \Bbb Q(\zeta_{\lambda})$. Furthermore, since $\mathscr{O}^{\times}$ is infinite by assumption, we can find a fundamental unit $u \in \mathscr{O}^{\times}$. By Lemma 4.1(3), it suffices to show that for any $(a , b) \in \mathscr{R}(\mathscr{O})$, we have
\begin{equation}\label{E:ZZZ8}
(a , b) \stackrel{8}{\Rightarrow} (1 , 0).
\end{equation}
First, applying Lemma 4.3 with $T = (S \setminus V^k_{\infty}) \cup V(\mu)$ and $n = \mu$, we see that there exist $\alpha \in \mathscr{O}_k$ and $r \in \mathscr{O}^{\times}$ such that
$$
V(\alpha) \cap (S \cup V(\mu)) = \emptyset \ \ \text{and} \ \ (a , b) \stackrel{1}{\Rightarrow} (\alpha r^{\mu} , b).
$$
Next, applying Lemma 4.4 to the last pair, we find $a' \in \mathscr{O}$ and a $\Bbb Q$-split principal prime ideal $\mathfrak{q}$ such that $v_{\mathfrak{q}} \notin S \cup V(\lambda) \cup V(\phi(a'\mathscr{O}))$ and $(\alpha r^{\mu} , b) \stackrel{3}{\Rightarrow} (a' , q^{\mu m})$ for any $m \equiv 1(\mathrm{mod}\: \phi(a'\mathscr{O})$. Then
\begin{equation}\label{E:ZZZ40}
(a , b) \stackrel{4}{\Rightarrow} (a' , q^{\mu m}) \ \ \text{for any} \ \ m \equiv 1(\mathrm{mod}\: \phi(a'\mathscr{O})).
\end{equation}

To proceed with the argument we will now specify $m$. We let $P$ and $Q$ denote the sets of prime divisors of $\lambda / \mu$ and $\phi(a'\mathscr{O})$, respectively, and define $\lambda'$ and $\mu'$ as in \S3.3; we note that by construction $\mathfrak{q}$ is relatively prime to $\lambda'$.  So, we can apply Theorem 3.7 which yields a $\Bbb Q$-split principal prime ideal $\mathfrak{p} = \pi \mathscr{O}$ so that $v_{\mathfrak{p}} \notin V(\phi(a'\mathscr{O}))$ and conditions (1) - (3) are satisfied.  Let $s = \ell_{\mathfrak{p}}(u)$ be the $\mathfrak{p}$-level of $u$.  Condition (3) implies that
$$
\mathrm{gcd}(\phi(\mathfrak{p})/\lambda , \lambda'/\lambda) = 1 = \mathrm{gcd}(\phi(\mathfrak{p})/\lambda , \phi(a'\mathscr{O}))
$$
since $\lambda'/\lambda$ is the product of all prime divisors of $\phi(a'\mathscr{O})$. It follows that the numbers $\phi(\mathfrak{p}^s)/\lambda$ and $\phi(a'\mathscr{O})$ are relatively prime, and therefore one can pick a positive integer $m$ so that
$$
m \equiv 0(\mathrm{mod}\: \phi(\mathfrak{p}^s)/\lambda) \ \ \text{and} \ \ m \equiv 1(\mathrm{mod}\: \phi(a'\mathscr{O})).
$$
Fix this $m$ for the rest of the proof.

Condition (2) of Theorem 3.7 enables us to apply Corollary 4.2 with $z = \pi$ and $t_0 = s$ to find $t \geq s$ so that $(a' , q^{\mu m}) \stackrel{1}{\Rightarrow} (\pi^t , q^{\mu m})$. Since $P$ consists of all prime divisors of $\lambda / \mu$, condition (1) of Theorem 3.7 implies that $\lambda/\mu$ divides the order of $u(\mathrm{mod}\: \mathfrak{p})$. Now, applying Lemma 4.5 with $\delta = \mu$ and $b = q^{\mu}$, we see that $(\pi^t , q^{\mu m}) \stackrel{1}{\Rightarrow} (\pi^t , u^{n_t})$ for some integer $n_t$. Finally, since $u$ is a unit, we have $(\pi^t , u^{n_t}) \stackrel{2}{\Rightarrow} (1 , 0)$. Combining these computations with (\ref{E:ZZZ40}), we obtain (\ref{E:ZZZ8}), completing the proof.
\end{proof}

\begin{cor}\label{T:Ma}
Assume that the group $\mathscr{O}^{\times}$ is infinite. Then for $n \geq 2$, any matrix $A \in \mathrm{SL}_n(\mathscr{O})$ is a product of $\leq \frac{1}{2}(3n^2 - n) + 4$ elementary matrices.
\end{cor}
\begin{proof}
For $n = 2$, this is equivalent to Theorem 1.1. Now, let $n \geq 3$. Since the ring $\mathscr{O}$ is Dedekind, it is well-known and easy to show that any $A \in \mathrm{SL}_n(\mathscr{O})$ can be reduced to a matrix in $\mathrm{SL}_2(\mathscr{O})$ by at most $\frac{1}{2}(3n^2 - n) - 5$ elementary operations (cf. [CK1, p. 683]). Now, our result immediately follows from Theorem 1.1.
\end{proof}

\noindent{\it Proof of Corollary \ref{c1}}.
Let $$
e_{+} \colon \alpha \mapsto \left( \begin{array}{cc} 1 & \alpha \\ 0 & 1 \end{array} \right) \ \ \text{and} \ \ e_{-} \colon \alpha \mapsto \left( \begin{array}{cc} 1 & 0 \\ \alpha & 1 \end{array} \right)
$$
be the standard 1-parameter subgroups. Set $U^{\pm} = e_{\pm}(\mathscr{O})$. In view of Theorem 1.1, it is enough to show that each of the subgroups $U^+$ and $U^-$ is contained in a product of finitely many cyclic subgroups of $\mathrm{SL}_2(\mathscr{O})$. Let $h_k$ be the class number of $k$. Then there exists $t \in \mathscr{O}^{\times}$ such that $v(t) = h_k$ for all $v \in S \setminus V_{\infty}^k$ and $v(t) = 0$ for all $v \notin S$. Then $\mathscr{O} = \mathscr{O}_k[1/t]$. So, letting $U_0^{\pm} = e_{\pm}(\mathscr{O}_k)$ and $h = \left( \begin{array}{cl} t & 0 \\ 0 & t^{-1} \end{array} \right)$, we will have the inclusion $$
U^{\pm} \subset \langle h \rangle \, U_0^{\pm} \, \langle h \rangle.
$$
On the other hand, if $w_1, \ldots , w_n$ (where $n = [k : \Bbb Q]$) is a $\Bbb Z$-basis of $\mathscr{O}_k$ then $U_0^{\pm} = \langle e_{\pm}(w_1) \rangle \cdots \langle e_{\pm}(w_n) \rangle$, hence
\begin{equation}\label{E:Incl7}
U^{\pm} \subset \langle h \rangle \langle e_{\pm}(w_1) \rangle \cdots \langle e_{\pm}(w_n) \rangle \langle h \rangle,
\end{equation}
as required.  \hfill$\square$

\begin{rem}
1. Quantitatively, it follows from the proof of Theorem 1.1 that $\mathrm{SL}_2(\mathscr{O}) = U^{-} U^{+} \cdots U^-$ (nine factors), so since the right-hand side of (23) involves $n + 2$ cyclic subgroups, with $\langle h \rangle$ at both ends, we obtain that $\mathrm{SL}_2(\mathscr{O})$ is a product of $9[k : \Bbb Q] + 10$ cyclic subgroups. Also, it  follows from \cite{Vs} that $\mathrm{SL}_2(\mathbb{Z}[1/p])$ is a product of 11 cyclic subgroups.

2. If $S = V^k_{\infty}$, then the proof of Corollary \ref{c1} yields a factorization of $\mathrm{SL}_2(\mathscr{O})$ as a finite product $\langle \gamma_1 \rangle \cdots \langle \gamma_d \rangle$ of cyclic subgroups where all generators $\gamma_i$ are elementary matrices, hence {\it unipotent}. On the contrary, when $S \neq V^k_{\infty}$, the factorization we produce involves some diagonal ({\it semisimple}) matrices. So, it is worth pointing out in the latter case there is no factorization with all $\gamma_i$ unipotent. Indeed, let $v \in S \setminus V^k_{\infty}$ and let $\gamma \in \mathrm{SL}_2(\mathscr{O})$ be unipotent. Then there exists $N = N(\gamma)$ such that for any $a = (a_{ij}) \in \langle \gamma \rangle$ we have $v(a_{ij}) \leq N(\gamma)$ for all $i , j \in \{1 , 2\}$. It follows that if $\mathrm{SL}_2(\mathscr{O}) = \langle \gamma_1 \rangle \cdots \langle \gamma_d \rangle$ where all $\gamma_i$ are unipotent, then there exists $N_0$ such that for any $a = (a_{ij}) \in \mathrm{SL}_2(\mathscr{O})$ we have $v(a_{ij}) \leq N_0$ for $i , j \in \{1 , 2\}$, which is absurd.


\end{rem}

\section{Example}
For a ring of $S$-integers $\mathscr{O}$ in a number field $k$ such that the group of units $\mathscr{O}^{\times}$ is infinite, we let $\nu(\mathscr{O})$ denote the smallest positive integer with the property that every matrix in $\mathrm{SL}_2(\mathscr{O})$ is a product of $\leq \nu(\mathscr{O})$ elementary matrices. So, the result of [Vs] implies that $\nu(\Bbb Z[1/p]) \leq 5$ for any prime $p$, and our Theorem 1.1 yields that $\nu(\mathscr{O}) \leq 9$ for any $\mathscr{O}$ as above. It may be of some interest to determine the exact value of $\nu(\mathscr{O})$ in some situations. In Example 2.1 on p. 289, Vsemirnov claims that the matrix $$
M=\left( \begin{array}{cc} 5 & 12 \\ 12 & 29 \end{array} \right)
$$
is not a product of four elementary matrices in $\mathrm{SL}_2(\Bbb Z[1/p])$ for any $p \equiv 1(\mathrm{mod}\: 29)$, and therefore $\nu(\Bbb Z[1/p]) = 5$ in this case. However this example is faulty because for any prime $p$, in $\mathrm{SL}_2(\Bbb Z[1/p])$ we have $$
M=\left( \begin{array}{cc} 5 & 12 \\ 12 & 29 \end{array} \right) =
\left(\ \left( \begin{array}{cc} 1 & 0 \\ 2 & 1 \end{array} \right) \cdot
\ \left( \begin{array}{cc} 1 & 2 \\ 0 & 1 \end{array} \right)\ \right)^2
$$
However, it turns out that the assertion that $\nu(\Bbb Z[1/p]) = 5$ is valid not only for
$p \equiv 1(\mathrm{mod}\: 29)$ but in fact for all $p > 7$. More precisely, we have the following.

\begin{prop}\label{asdf}
Let $\Cal O=\Bbb Z$$[1/p]$, where $p$ is prime $>7$.  Then not every matrix in ${\rm SL}_2(\Cal O)$
is a product of four elementary matrices.
\end{prop}
In the remainder of this section, unless stated otherwise, we will work with congruences
over the ring $\mathscr{O}$ rather than $\Bbb Z$, so the notation $a \equiv b(\mathrm{mod}\: n)$
means that elements $a , b \in \mathscr{O}$ are congruent modulo the ideal $n\mathscr{O}$.
We begin the proof of the proposition with the following lemma.
\begin{Lem}\label{fdsa}
Let $\Cal O=\Bbb Z$$[1/p]$, where $p$ is any prime, and let $r$ be a positive integer
satisfying $p\equiv 1({\rm mod\ }r)$.  Then any matrix $A\in{\rm SL}_2(\Cal O)$ of the form
\begin{equation}\label{ffff}
A=\left( \begin{array}{cc} 1-p^\alpha & * \\ * & 1-p^\beta \end{array} \right),
\ \alpha,\beta\in\Bbb Z
\end{equation}
which is a product of four elementary matrices, satisfies the congruence $$
A\equiv\pm\left(\begin{array}{cc}0&1\\-1&0\end{array}\right)({\rm mod}\:r).
$$
\end{Lem}\begin{proof}
We note right away that the required congruence is obvious for the diagonal entries, so
we only need to establish it for the off-diagonal ones.  Since $A$ is a product of four
elementary matrices, it admits one of the following presentations:
\begin{equation}\label{E:ZZZZ1}
A=E_{12}(a)E_{21}(b)E_{12}(c)E_{21}(d),
\end{equation} or
\begin{equation}\label{E:ZZZZ2}
A=E_{21}(a)E_{12}(b)E_{21}(c)E_{12}(d),
\end{equation}
with $a,b,c,d\in\Cal O$.

First, suppose we have (\ref{E:ZZZZ1}). Then
$$
A = \left( \begin{array}{cc}  * & * \\ * & 1 + bc \end{array}  \right).
$$
Comparing with (\ref{ffff}), we get $bc = -p^{\beta}$, so $b$ and $c$ are powers of $p$ with opposite signs.
Thus, $A$ looks as follows:$$
A=E_{12}(a)E_{21}(\pm p^\gamma)E_{12}(\mp p^\delta)E_{21}(d)=
\left( \begin{array}{cc} * & a(1-p^{\gamma+\delta})\mp p^\delta \\
d(1-p^{\gamma+\delta})\pm p^\gamma & *\end{array} \right).
$$
Consequently, the required congruences for the off-diagonal entries immediately follow
from the fact that $p \equiv 1(\mathrm{mod}\: r)$, proving the lemma in this case.

Now, suppose we have (\ref{E:ZZZZ2}).  Then $$
A^{-1}=E_{12}(-d)E_{21}(-c)E_{12}(-b)E_{21}(-a),
$$
which means that $A^{-1}$ has a presentation of the form (\ref{E:ZZZZ1}).  Since
the required congruence in this case has already been established, we conclude that $$
A^{-1}\equiv\pm\left( \begin{array}{cc} 0 & -1 \\ 1 & 0 \end{array} \right)({\rm mod}\:r).
$$
But then we have $$
A\equiv\pm\left( \begin{array}{cc} 0 & 1 \\ -1 & 0 \end{array} \right)({\rm mod}\:r),
$$
as required.
\end{proof}

To prove the proposition, we will consider two cases.

\vskip1mm

{\sc Case 1.} {\it $p-2$ is composite.}  Write $p-2 = r_1\cdot r_2$, where $r_1,r_2$ are
positive integers $>1$, and set $r=p-1$.  Then
\begin{equation}\label{E:XXX2}
r_i \not\equiv \pm 1 (\mathrm{mod}\: r) \ \ \text{for} \ \ i = 1, 2.
\end{equation}
Indeed, we can assume that $r_2\leq\sqrt{p-2}$.  If $r_2\equiv\pm 1(\mathrm{mod}\: r)$ then because $r$ is prime to $p$, the number
$r_2 \mp 1$ would be a nonzero integral multiple of $r$. Then $r \leq r_2 + 1$, hence
$$p-2\leq\sqrt{p-2} + 1.$$  But this is impossible since $p>3$.  Thus, $r_2\not\equiv\pm 1({\rm mod}\:r)$.
Since $r_1 \cdot r_2 \equiv -1 (\mathrm{mod}\: r)$, condition (\ref{E:XXX2}) follows.

Now, consider the matrix $$
A=\left( \begin{array}{cc} 1-p & r_1\cdot p \\ r_2 & 1-p \end{array} \right)
$$
One immediately checks that $A\in{\rm SL}_2(\Cal O)$.  At the same time, $A$ is of the form
(\ref{ffff}).  Then Lemma \ref{fdsa} in conjunction with (\ref{E:XXX2}) implies that $A$ is
not a product of four elementary matrices.

\vskip1mm

{\sc Case 2.} {\it $p$ and $p-2$ are both primes.}  In the beginning of this paragraph we will use
congruences in $\Bbb Z$.  Clearly, a prime $> 3$ can only be congruent to
$\pm 1(\mathrm{mod}\: 6{\Bbb Z})$.  Since $p > 5$ and $p-2$ is also prime, in our situation
we must have $p \equiv 1(\mathrm{mod}\: 6{\Bbb Z})$.  Furthermore, since $p>7$, the congruence
$p\equiv 0$ or $2(\mathrm{mod} \: 5{\Bbb Z})$ is impossible.  Thus, in the case at hand we have $$
p\equiv 1,13, \, \text{or} \ 19(\mathrm{mod}\: 30{\Bbb Z}).
$$
If $p\equiv 13(\mathrm{mod}\: 30{\Bbb Z})$, then $p^3\equiv 7(\mathrm{mod}\: 30{\Bbb Z})$, and therefore $p^3-2$
is an integral multiple of $5$.  Set $r=p-1$ and $s=(p^3-2)/5$, and consider the matrix$$
A=\left( \begin{array}{cc} 1-p^3 & 5p^3 \\ s & 1-p^3 \end{array} \right)
$$
Then $A$ is a matrix in ${\rm SL}_2(\Cal O)$ having form (\ref{ffff}).  Note that
$5p^3\equiv 5(\mathrm{mod}\: r)$, which is different from $\pm 1(\mathrm{mod}\: r)$ since $r>6$.
Now, it follows from Lemma \ref{fdsa} that $A$ is not a product of four elementary matrices.

It remains to treat the case where $p \equiv 1$ or $19(\mathrm{mod}\: 30{\Bbb Z})$.
Consider the following matrix:$$
A=\left( \begin{array}{cc} 900 & 53\cdot 899 \\ 17 & 900 \end{array} \right),
$$
and note that $A \in \mathrm{SL}_2(\Bbb Z)$ and $$
A^{-1}=\left( \begin{array}{cc} 900 & -53\cdot 899 \\ -17 & 900 \end{array} \right).
$$

It suffices to show that neither $A$ nor $A^{-1}$ can be written in the form
\begin{equation}\label{E:TTT1}
E_{12}(a)E_{21}(b)E_{12}(c)E_{21}(d) = \left( \begin{array}{cc} * & c + a(1+bc) \\
b + d(1+bc) & (1+bc) \end{array} \right),{\rm\ with\ }a,b,c,d\in\Cal O
\end{equation}
Set $$
t = b + d(1 + bc) \ \ \text{and} \ \ u = c + a(1 + bc).
$$
Assume that either $A$ or $A^{-1}$ is written in the form (\ref{E:TTT1}).
Then $1 + bc = 900$, so $$
b,c\in\{\pm p^n,\,\pm 29p^n,\,\pm 31p^n,\,\pm 899p^n\mid n\in\Bbb Z \}.
$$
On the other hand, we have the following congruences in $\mathscr{O} = {\Bbb Z}[1/p]$:
$$
t \equiv b (\mathrm{mod}\: 30) \ \ \text{and} \ \ u \equiv c(\mathrm{mod}\: 30).
$$
Analyzing the above list of possibilities for $b$ and $c$, we conclude that
each of $t$ and $u$ is $\equiv \pm p^n(\mathrm{mod}\: 30)$ for some integer $n$.
Thus, if $p \equiv 1(\mathrm{mod}\: 30)$ then $t , u \equiv \pm 1(\mathrm{mod}\: 30)$,
and if $p \equiv 19(\mathrm{mod}\: 30)$ then $t , u \equiv \pm 1, \pm 19(\mathrm{mod}\: 30)$.
Since $17 \not\equiv \pm 1, \pm 19 (\mathrm{mod}\: 30)$, we obtain a contradiction in either case.
(We observe that the argument in this last case is inspired by Vsemirnov's argument in his Example 2.1.)

\vskip5mm

\noindent {\small {\bf Acknowledgements.} During the preparation of the final version of the paper the second author visited Princeton University
and the Institute for Advanced Study on a Simons Fellowship. The hospitality of both institutions and the generous support of the Simons Foundation are thankfully acknowledged. The third author gratefully acknowledges fruitful conversations with Maxim Vsemirnov. We thank Dong Quan Ngoc Nguyen for his comments regarding Theorem \ref{dit}. Our most sincere thanks are due to the anonymous referee who has read the paper extremely carefully and offered a number of  corrections and valuable suggestions.}

\vskip3mm

\end{document}